\documentclass[12pt,letterpaper, oneside, reqno]{amsart}
\usepackage{lmodern}
\usepackage{amssymb,amsthm,amsmath,amstext,amsxtra, amscd, amsrefs}
\usepackage{enumitem}
\usepackage{mathrsfs}
\usepackage{comment}
\usepackage{geometry}
\usepackage{fullpage}
\usepackage{colonequals}
\usepackage[colorlinks=true, linkcolor=blue, citecolor=blue,urlcolor=blue]{hyperref}
\usepackage{cleveref}
\numberwithin{equation}{subsection}
\usepackage{tikz-cd}
\usepackage{mathtools}
\usepackage{mathrsfs}
\usepackage{textgreek}
\allowdisplaybreaks

\newtheorem{thm}[equation]{Theorem}

\newtheorem{prop}[equation]{Proposition}
\newtheorem{lem}[equation]{Lemma}
\newtheorem{cor}[equation]{Corollary}
\theoremstyle{definition}
\newtheorem{definition}[equation]{Definition}

\theoremstyle{remark}
\newtheorem{remark}[equation]{Remark}
\newtheorem{exm}[equation]{Example}
\newtheorem{example}[equation]{Example}

\newcommand{\Z}{\mathbb{Z}}  
\newcommand{\Q}{\mathbb{Q}}

\newcommand{\G}{\mathbb{G}}

\DeclareMathOperator{\GL}{GL}
\DeclareMathOperator{\disc}{disc}

\DeclareMathOperator{\Gen}{Gen}
\DeclareMathOperator{\Cls}{Cls}
\DeclareMathOperator{\opchar}{char} 
\DeclareMathOperator{\Hom}{Hom}

\DeclareMathOperator{\rk}{rank}
\DeclareMathOperator{\Cl}{Cl}

\DeclareMathOperator{\SL}{SL}
\DeclareMathOperator{\Nm}{Nm}
\DeclareMathOperator{\Tr}{Tr}
\DeclareMathOperator{\Clf}{Clf}

\DeclareMathOperator{\End}{End}

\DeclareMathOperator{\M}{M}

\DeclareMathOperator{\Sym}{Sym}
\DeclareMathOperator{\id}{id}
\DeclareMathOperator{\Pic}{Pic}
\DeclareMathOperator{\Gal}{Gal}
\DeclareMathOperator{\Frac}{Frac}

\DeclareMathOperator{\nrd}{nrd}
\DeclareMathOperator{\Ann}{Ann}
\DeclareMathOperator{\GO}{GO}
\DeclareMathOperator{\Orthogonal}{O}
\DeclareMathOperator{\GSO}{GSO}
\DeclareMathOperator{\SO}{SO}

\DeclareMathOperator{\Aut}{Aut}

\DeclareMathOperator{\RigPic}{RigPic}

\DeclareMathOperator{\Spec}{Spec}

\DeclareMathOperator{\ev}{ev}
\DeclareMathOperator{\univ}{univ}
\newcommand{\Gm}{\mathbb{G}_m}

\newcommand{\tbigwedge}{\smash{\raisebox{0.2ex}{\ensuremath{\textstyle{\bigwedge}}}}}

\newcommand{\scrA}{\mathscr{A}}
\newcommand{\scrC}{\mathscr{C}}
\newcommand{\scrE}{\mathscr{E}}
\newcommand{\scrF}{\mathscr{F}}
\newcommand{\scrO}{\mathscr{O}}
\newcommand{\scrM}{\mathscr{M}}
\newcommand{\scrL}{\mathscr{L}}
\newcommand{\calO}{\scrO}
\newcommand{\scrI}{\mathscr{I}}
\newcommand{\scrJ}{\mathscr{J}}
\newcommand{\scrN}{\mathscr{N}}
\newcommand{\scrP}{\mathscr{P}}
\newcommand{\scrZ}{\mathscr{Z}}

\newcommand{\frakp}{\mathfrak{p}}

\DeclareMathOperator{\HHom}{\mathscr{H}\!\mathit{om}}
\DeclareMathOperator{\TTen}{\mathscr{T}\!\mathit{en}}

\newenvironment{enumalph}
{\begin{enumerate}}
{\end{enumerate}}

\newenvironment{enumroman}
{\begin{enumerate}}
{\end{enumerate}}

\newcommand{\olsi}[1]{\,\overline{\!{#1}}}

\newcommand{\defi}[1]{\textit{\textsf{#1}}}

\newcommand{\catPic}{\textup{\textbf{\textsf{Pic}}}}
\newcommand{\catSch}{\textup{\textbf{\textsf{Sch}}}}
\newcommand{\catAlg}{\textup{\textbf{\textsf{Alg}}}}
\newcommand{\catQuad}{\textup{\textbf{\textsf{Alg}}}_2}
\newcommand{\catQMod}{\textup{\textbf{\textsf{QMod}}}}
\newcommand{\catBQMod}{\textup{\textbf{\textsf{QMod}}}_2}
\newcommand{\catPSReg}{\textup{\textbf{\textsf{PsReg}}}}
\newcommand{\catPsReg}{\catPSReg}
\newcommand{\catBPSReg}{\textup{\textbf{\textsf{PsReg}}}_2}
\newcommand{\catGoodBPSReg}{\textup{\textbf{\textsf{PsReg}}}_2^{\textup{gfr}}}

\newcommand{\Qfr}{\catQMod^{\textup{fr}}}

\author{John Voight}
\address{Department of Mathematics, Dartmouth College, Kemeny Hall, Hanover, NH 03755, USA; Carslaw Building (F07), Department of Mathematics and Statistics, University of Sydney, NSW 2006, Australia}
\email{jvoight@gmail.com}

\author{Haochen Wu}
\address{Department of Mathematics, Dartmouth College, Kemeny Hall, Hanover, NH 03755, USA}
\email{haochen.wu.gr@dartmouth.edu}

\title{An orthogonal perspective on Gauss composition}

\begin{document}

\begin{abstract}
We revisit Gauss composition over a general base scheme, with a focus on orthogonal groups.  We show that the Clifford and norm functors provide an equivalence of stacks between binary quadratic modules and pseudoregular modules over quadratic algebras.  As a consequence, we exhibit a composition law for coprimitive forms over a general base, including a universal version of Dirichlet composition.  This perspective synthesizes the constructions of Kneser and Wood, reconciles algebraic and geometric approaches, and clarifies the role of orientations and the natural emergence of narrow class groups.  
\end{abstract}

\setcounter{tocdepth}{1}


\maketitle
\tableofcontents

\section{Introduction}

\subsection{Motivation}

Classical Gauss composition relates primitive binary quadratic forms of fixed discriminant to ideal or narrow ideal classes in the corresponding quadratic order.  Many papers have been written providing generalizations of this important result (see \cref{recentwork} for recent work).  The purpose of this paper is to explain these variants from a single functorial Clifford--norm equivalence over the stack of quadratic algebras: orientations, rigidifications, similarities, and isometries are obtained by taking appropriate $2$-fibers and changing the allowed morphisms.  As an application, we exhibit a general composition law including a universal form of Dirichlet composition.


\[
\begin{array}{c|c|c}
\text{binary quadratic side} & \text{algebra side} & \text{primitive classes} \\ \hline\hline
\GO, \text{similarities} & \text{semilinear isomorphisms} & \Pic Y/\Aut_X(Y) \\
\GSO, \text{oriented similarities} & \text{isomorphisms} & \Pic Y \\ \hline
\Orthogonal, \text{isometries} & \text{rigidified semilinear isomorphisms} & \RigPic Y/\Aut_X(Y) \\
\SO, \text{oriented isometries} & \text{rigidified isomorphisms} & \RigPic Y 
\end{array}
\]

\subsection{Setup}

To state our results, we will first need some definitions and notation; for further details in the setup, see \cref{sec:setup}.  

Let $X$ be a scheme.  A \defi{quadratic $\scrO_X$-module} is a locally free $\scrO_X$-module $\scrM$ of finite rank equipped with a quadratic map $Q \colon \scrM \to \scrL$ taking values in an invertible $\scrO_X$-module $\scrL$, called the \defi{value module}.  A quadratic map is equivalently specified by a global section of $(\Sym_2 \scrM)^\vee \otimes \scrL$. If we further take $\scrL=\scrO_X$, then we call $Q$ a \defi{quadratic form}. A \defi{similarity} between two quadratic $\scrO_X$-modules is the natural commutative square \eqref{eqn:commdiag}.
The fibered category of quadratic modules of rank $m$ under similarities is naturally a quotient stack $\catQMod_m \simeq [\mathbb{A}^{m(m+1)/2}/(\GL_m \times \Gm)]$.  

The Clifford functor associates to a quadratic $\scrO_X$-module $Q \colon \scrM \to \scrL$ the odd Clifford bimodule $\Clf^1(Q)$ over the even Clifford $\scrO_X$-algebra $\Clf^0(Q)$.  The Clifford bimodule inherits the following key structural property.  Let $\scrA$ be a locally free $\scrO_X$-algebra of rank $n$, and let $\scrI$ be a left $\scrA$-module which is locally free of rank $n$ as an $\scrO_X$-module.  Recall that the \emph{regular} $\scrA$-module is $\scrA$ itself (considered as an $\scrA$-module).  Accordingly, we say that $\scrI$ is \defi{pseudoregular} if every algebra element acts with the same characteristic polynomial as in the regular representation: for all open $U \subseteq X$ and all $\alpha \in \scrA(U)$, the characteristic polynomials of $\alpha$ acting by left multiplication on $\scrA(U)$ and $\scrI(U)$ are equal.  We prove (\ref{prop:cliffpseudo}) that $\Clf^1(Q)$ is pseudoregular as a left (and right) $\Clf^0(Q)$-module.  When $n=2$, it is enough to check traces (\Cref{prop:traceable}), so in this case $\scrI$ is also called \defi{traceable} by Wood \cite{Wood}*{p.~1758}.  

We now restrict to the binary case $m=n=2$.  Then $\Clf^0(Q)$ is a \defi{quadratic} $\scrO_X$-algebra, i.e., one that is locally free of rank $2$ as an $\scrO_X$-module.  Quadratic algebras form a quotient stack $\catQuad \simeq [\mathbb{A}^2/(\G_a \rtimes \Gm)]$, and the formation of the even Clifford algebra defines a morphism of stacks 
\begin{equation}
\Clf^0 \colon \catBQMod \to \catQuad.
\end{equation}

Norms furnish an inverse to the Clifford functor, as follows.  Let $\scrO_Y$ be a quadratic $\scrO_X$-algebra and let $\scrI$ be a pseudoregular $\calO_Y$-module.  We define a canonical \defi{norm map} 
\begin{equation} 
N_\scrI \colon \scrI \to \scrN(\scrI)
\end{equation}
where the codomain is
\begin{equation} \label{eqn:normmodule}
\scrN(\scrI) \colonequals \tbigwedge^2 \scrI \otimes (\scrO_Y/\scrO_X)^\vee;
\end{equation}
it earns its name from the property that 
\begin{equation}
N_\scrI(\gamma x)=\Nm(\gamma)N_\scrI(x)
\end{equation}
for all $x \in \scrI(U)$ and $\gamma \in \scrO_Y(U)$ on open $U \subseteq X$, matching Kneser \cite{Kneser}*{section 4}. The norm map arises from a twist of the \defi{canonical exterior form}, the unique quadratic map
\begin{equation} 
E_\scrI \colon \scrI \otimes \scrO_Y/\scrO_X \to \tbigwedge^2 \scrI \otimes \scrO_Y/\scrO_X
\end{equation}
such that $E_\scrI(x \otimes \gamma) = (\gamma x \wedge x) \otimes \gamma$ for all $x \in \scrI(U)$ and $\gamma \in (\scrO_Y/\scrO_X)(U)$ (independent of lift).  After appropriate identifications and (perhaps a bit surprisingly) a duality, we show this agrees with the construction of Wood \cite{Wood}*{Theorem 2.1} in \cref{sec:wood}; see also Mondal--Venkata Balaji \cite{Mondal}*{Theorem 5.4}.

\subsection{Similarity classes and orientations}

Our first theorem corresponds to the reductive group $\GO$.  Let $\catBPSReg$ be the category of pseudoregular modules over quadratic algebras under semilinear isomorphisms, with the natural forgetful map $\catBPSReg \to \catQuad$ to quadratic algebras: for more, see \Cref{def:pseudoregcat}.

\begin{thm} \label{Main Thm GO}
The Clifford and norm functors define an equivalence of stacks  
\begin{equation}  \label{catbqmodbpsreg}
\catBQMod \xrightarrow{\sim} \catBPSReg 
\end{equation}
over $\catQuad$, namely between
\begin{center}
binary quadratic modules under similarities
\end{center}
and
\begin{center}
pseudoregular modules over quadratic algebras under semilinear isomorphisms.
\end{center}
This equivalence preserves discriminants of objects.   
\end{thm}

In particular, the Clifford and norm functors are quasi-inverse.  Wood \cite{Wood}*{Theorem 2.1} defines a functor from pseudoregular (traceable) modules to binary quadratic modules and shows it induces an equivalence without Clifford or norm language.  
Mondal--Venkata Balaji \cite{Mondal}*{Theorem 3.1} also proved \Cref{Main Thm GO} as an equivalence of categories.

We now examine the fibers in \Cref{Main Thm GO} over $\catQuad$ (for the stacky setup, see \cref{sec:fibers}).  For $Q \colon \scrM~\to~\scrL$ a binary quadratic $\scrO_X$-module, an \defi{$\scrO_Y$-orientation} of $Q$ is an $\scrO_X$-algebra isomorphism $\Clf^0(Q) \xrightarrow{\sim} \scrO_Y$.  A similarity-preserving orientation is called an \defi{oriented similarity}.  
The $2$-fiber of $\catBQMod \to \catQuad$ over $\scrO_Y$ is canonically the category of $\scrO_Y$-oriented binary quadratic $\scrO_X$-modules under oriented similarities.  The following result is then immediate from \Cref{Main Thm GO}, corresponding to the group $\GSO$ (and proven in \cref{sec:orient}).  

\begin{cor} \label{Main Thm GSO}
Let $\scrO_Y$ be a quadratic $\scrO_X$-algebra.  Then the Clifford and norm functors define an equivalence of categories between 
\begin{center}
$\scrO_Y$-oriented binary quadratic $\scrO_X$-modules under oriented similarities
\end{center}
and
\begin{center}
pseudoregular $\scrO_Y$-modules under $\scrO_Y$-module isomorphisms.
\end{center}
\end{cor}

In particular, if $\scrM \leftrightarrow \scrI$ in the equivalence, then we have an isomorphism of group schemes $\GSO(Q) \simeq \Aut_{\scrO_Y}(\scrI)$ over $X$.  
Moreover, the equivalence in \Cref{Main Thm GSO} commutes with a Cartesian base change $Y' \to Y$ over $X' \to X$ under pullback.  

We say a quadratic $\scrO_X$-module $Q$ is \defi{primitive} if values of $Q$ generate $\scrL$ on every open set.  Primitive $\scrO_Y$-oriented binary quadratic $\scrO_X$-modules correspond to invertible (i.e., locally principal) $\scrO_Y$-modules under $\scrO_Y$-module isomorphisms (\Cref{primitive = invertible}).  
We obtain the following corollary.  

\begin{cor} \label{GSO class =Pic Y}
Let $\calO_Y$ be a quadratic $\calO_X$-algebra. Then the Clifford and norm maps define mutually inverse bijections between the set of oriented similarity classes of \emph{primitive} $\calO_Y$-oriented binary quadratic $\calO_X$-modules and the group $\Pic Y$.  
\end{cor}

\Cref{GSO class =Pic Y} provides a composition law on the set of oriented similarity classes of primitive binary quadratic modules, spelled out explicitly in the next section.  This law is functorial, i.e., for $Y' \to Y$ it commutes with pullback $\Pic Y \to \Pic Y'$.  Although Wood \cite{Wood} did not make it explicit, other authors \cites{O'Dorney,Dallaporta,Mondal} have adapted her equivalence of categories to provide a composition law as in \Cref{GSO class =Pic Y} using a different notion of orientation that we call a \emph{rigidification} and pursue in \cref{sec:simlattices}: see also \cref{sec:contribution} for comparison.

Forgetting orientations, we also have bijections between the set of similarity classes of primitive binary quadratic $\calO_X$-modules $Q$ such that $\Clf^0(Q) \simeq \scrO_Y$ and the set $\Pic Y/\Aut_X(Y)$, where $\Aut_X(Y)$ acts on $\Pic Y$ by (semilinear) pullback; this can also be recovered directly from \Cref{Main Thm GO}.


\subsection{Universal composition law}

Next, we exhibit a partially defined universal composition law on binary quadratic modules.

Working locally, let $S$ be a free quadratic $R$-algebra, and let $Q,Q' \colon R^2 \to R$ be free, $S$-oriented, binary quadratic $R$-modules.  Then $Q(x,y)=ax^2+bxy+cy^2=[a,b,c]$ with $a,b,c \in R$ and similarly $Q'=[a',b',c']$.  Comparing $S$-orientations gives unique elements $r\in R$ and $u \in R^\times$ such that $b'=u(b+2r)$ and $a'c'=u^2(ac+br+r^2)$.  We say $Q,Q'$ are \defi{coprimitive} if $\langle a,b,c,a',b',c',r \rangle = R$; the notion is independent of the choice of basis, and thus it extends globally.  Of course, if either $Q$ or $Q'$ is primitive, then $Q,Q'$ are coprimitive.  

\begin{thm}\label{thm: composition law}
Let $\scrO_Y$ be a quadratic $\scrO_X$-algebra.  Then the set of oriented similarity classes of $\scrO_Y$-oriented binary quadratic $\scrO_X$-modules carries a functorial, commutative partial monoid structure $[Q''] = [Q] \ast [Q']$ defined whenever $Q,Q'$ are coprimitive.  
\end{thm}

The key observation in the proof of \Cref{thm: composition law} (see \cref{sec:tensorcomp}) is that the tensor product of pseudoregular modules over a quadratic algebra furnishes a composition law when the result is pseudoregular, and this can be precisely understood through coprimitivity (\Cref{tensor product is pseudoregular}).  

We recover Dirichlet composition of united forms in a \emph{universal} way as follows (\Cref{cor:dirichletcomp}).  Again in the local case, suppose we have $\langle a,a',b+r \rangle = R$ (so in particular, $Q,Q'$ are coprimitive).  Let $s,s',t \in R$ be such that 
\begin{equation}
s a+s' a'+t (b+r)=1.
\end{equation}
Then $[Q] \ast  [Q'] = [Q'']$ where 
\begin{equation}
Q''=[ aa', b-2a(ct-rs), cs'+(s b+ct)(c't+rs')+as(cs'+c's)].
\end{equation}
Over $R=\Z$, these coefficients satisfy the congruence conditions giving rise to classical Dirichlet composition!

\subsection{Rigidifications and isometry classes} \label{sec:simlattices}

In the study of quadratic modules, often the value module is fixed.  For example, the more basic notion of a \defi{quadratic form} corresponds to the case where the value module is trivial, $\scrL=\scrO_X$.  

Accordingly, we define a \defi{similitude} between $Q \colon \scrM \to \scrL$ and $Q' \colon \scrM' \to \scrL'$ to be a similarity with $\scrL=\scrL'$, and an \defi{isometry} is a similitude with similitude factor $\lambda=1$.  To make this categorical, the value bundle map $Q \mapsto \scrL$ defines a morphism of stacks
\[ \catBQMod \to \catPic \]
where $\catPic=B\Gm$ is the stack of invertible modules under isomorphisms.  Together with even Clifford, we obtain a stack morphism
\begin{equation} 
\catBQMod \to \catQuad \times \catPic.
\end{equation}

We similarly define such a structure morphism for the stack of pseudoregular modules over quadratic algebras by taking $\scrI \mapsto \scrN(\scrI)$.  We then upgrade \Cref{Main Thm GO} to the following statement.

\begin{prop}{\label{Main Thm sim}}
The equivalence $\catBQMod \xrightarrow{\sim} \catBPSReg$ in \eqref{catbqmodbpsreg}
factors over the stack $\catQuad \times \catPic$.  
\end{prop}

We again look at $2$-fibers.  Let $\scrL$ be an invertible $\scrO_X$-module.  An $\scrL$-\defi{rigidification} of a pseudoregular $\scrO_Y$-module $\scrI$ is an isomorphism $r \colon \scrN(\scrI) \xrightarrow{\sim} \scrL$.  Wood \cite{Wood}*{section 5} calls this an \emph{$\scrL$-type orientation}, but we avoid that name to distinguish it from the orientation defined above.  If $\scrI, \scrI'$ are rigidified $\scrO_Y$-modules, we say that an $\scrO_Y$-module isomorphism $\varphi \colon \scrI \to \scrI'$ is \defi{rigidified} if the induced isomorphism $\scrN(\scrI)\to \scrN(\scrI')$ commutes with the rigidifications.  From a rigidification $r$ of $\scrI$, we define the \defi{modified norm map} by $r \circ N_{\scrI}$.  

\begin{cor} \label{Main Thm SO}
Let $\scrO_Y$ be a quadratic $\scrO_X$-algebra, and let $\scrL$ be an invertible $\scrO_X$-module.  Then the Clifford and modified norm functors define an equivalence of categories between 
\begin{center}
$\scrO_Y$-oriented binary quadratic $\scrO_X$-modules with value module $\scrL$ under oriented \emph{isometries}
\end{center}
and
\begin{center}
$\scrL$-rigidified pseudoregular $\scrO_Y$-modules under \emph{rigidified} $\scrO_Y$-module isomorphisms.
\end{center}
\end{cor}

Let $\RigPic_{\scrL} Y$ be the set of invertible $\scrL$-rigidified $\scrO_Y$-modules under rigidified $\scrO_Y$-module isomorphisms, when nonempty, a torsor under the group $\RigPic_{\scrO_X} Y$.  Then for $\scrL=\scrO_X$, \Cref{Main Thm SO} yields a bijection between $\scrO_Y$-oriented isometry classes of primitive binary quadratic \emph{forms} and the group $\RigPic_{\scrO_X} Y$.  We then have the following result (\Cref{prop: fiber of Rigpic}).

\begin{prop} \label{prop:rig}
Forgetting the rigidification fits into a short exact sequence 
\begin{equation} \label{eqn:narrowclass}
1 \to \scrO_X(X)^\times/\Nm \scrO_Y(Y)^\times
\to \RigPic_{\scrO_X} Y \to \Pic_{\scrO_X} Y \to 1,
\end{equation}
where $\Pic_{\scrO_X} Y \colonequals \ker \scrN$.
\end{prop}

More generally, the \defi{rigidified Picard set}
\begin{equation}
\RigPic Y \colonequals \bigsqcup_{[\scrL] \in \Pic X} \RigPic_{\scrL} Y
\end{equation}
is well-defined up to bijection over $\Pic Y$.  The forgetful map $\RigPic Y \to \Pic Y$ similarly has fibers which are torsors under the group $\scrO_X(X)^\times/\Nm \scrO_Y(Y)^\times$ (rescaling the rigidification).  Thus we get an exact sequence of pointed sets
\begin{equation} \label{eqn:exsequpotgrp}
1 \to \scrO_X(X)^\times/\Nm \scrO_Y(Y)^\times \to \RigPic Y \to \Pic Y \to 1.
\end{equation}

\begin{remark}
Choosing a monoidal skeleton for $\Pic X$, we may give $\RigPic Y$ the structure of a \emph{group} such that \eqref{eqn:exsequpotgrp} defines an exact sequence of groups; however, this group structure and the corresponding extension may depend on the choice of skeleton.
\end{remark}

There are analogous notions obtained by forgetting $\scrO_Y$-orientations, corresponding to the usual orthogonal group $\Orthogonal$. More generally (see \cref{sec: rigidified Picard group}), we may also choose a subgroup $H \leq \scrO_X(X)^\times$ and work with the subgroup of similitudes with similarity factor in $H$, interpolating between the similitude group and the isometry group. 


\subsection{Applications to lattices, genera, and class sets}
Specializing the preceding results to Dedekind domains recovers the classical ideal-theoretic classification of binary quadratic lattices. Let $R$ be a Dedekind domain with field of fractions $F$. Let $Q\colon V\to F$ be a nondegenerate binary quadratic form, and suppose $K\colonequals \Clf^0(Q)$ is a quadratic field. Fix an $R$-order $S\subset K$. We obtain the following result. 

\begin{cor}
The Clifford and norm functors define mutually inverse discriminant-preserving bijections as follows:
\begin{enumalph}
\item between the set of similarity classes of $R$-lattices $M \subset V$ with multiplicator ring $S$ and the group $\Pic S$; and
\item between the set of isometry classes of integral $R$-lattices $M \subset V$ with multiplicator ring $S$ and $Q(M)R=R$ and the group $\RigPic_R S$. 
\end{enumalph}
\end{cor}

The distinction between similarities and isometries is reflected by the passage from $\Pic S$ to $\RigPic_R S$ via the exact sequence \eqref{eq: Rigpic H exact seq}.
For $S=\Z_K$, this recovers the narrow class group in the real quadratic case and the corresponding signed variant in the imaginary quadratic case. 




\subsection{Contribution} \label{sec:contribution}

The first contribution of this paper is to synthesize the work of Kneser \cite{Kneser} and Wood \cite{Wood}, also independently pursued by Mondal--Venkata Balaji \cite{Mondal}.  Kneser approaches composition laws via the even Clifford algebra and the norm without a categorical equivalence, whereas Wood provides a categorical equivalence and geometric approach without a composition law and without Clifford and the norm.  Our theorems combine and recover their results.  In particular, we show that the functor defined by Wood is naturally isomorphic to the norm functor composed with a duality.  For further context on other related work, see \cref{recentwork}.  As a consequence, by viewing Gauss composition over an arbitrary base through the Clifford functor, we can see it as just one instance of the exceptional isomorphisms of Lie groups in low rank $n$ \cite{bookofinv}*{\S 15}: indeed, this paper fills in a gap in the cases of rank $n=1$ (monoid of quadratic rings, \cite{Voight:quadmonoid}) and $n=3$ (quaternion rings, \cite{Voight:quatring}).  

Second, this paper reconciles notions of orientation through a simple category-theoretic lens of taking the $2$-fiber.  Our approach was inspired by Kneser, who calls it \emph{$C$-type} \cite{Kneser}*{section 4}; this notion also appears in Knus--Merkurjev--Rost--Tignol \cite{bookofinv}*{Definition (12.40)} in the case of odd dimension over a field and in Auel \cite{Auel11}*{1.10} in the \'etale case.  This notion of orientation views the even Clifford algebra as a refinement of the discriminant in the classical case and resolves the problems caused by automorphisms of the quadratic algebra---this is necessary to get a composition law, as noticed by other authors.  This notion of orientation seems particularly well-suited to this task: the set of $\scrO_Y$-orientations on a binary quadratic module, if nonempty, is a torsor under $\Aut_{X}(Y)$, so for example if $Y \to X$ is \'etale and $X$ is connected, then $\Aut_{X}(Y) \simeq \Z/2\Z$.  By contrast, the set of rigidifications is a torsor for $\Gm(X)=\scrO_X(X)^\times$, controlling the value bundle.  For a comparison with other approaches, see \cref{sec:sorryhereitis}.   

Our third contribution is the extension of the composition law to the coprimitive case, in particular recovering classical Dirichlet composition over a general base. 

Our fourth contribution is to center the quartet of orthogonal groups $\GO$, $\GSO$, $\SO$, and $\Orthogonal$, thereby fully recovering Gauss composition.  This identifies Wood's \emph{linear binary quadratic forms} under the action of $\GL_1 \times \GL_2$ as binary quadratic forms under similarity.  In particular, previous work generalizing Gauss composition does not address the role of the narrow class group; we see it here as arising from oriented isometries.  

Finally, we generalize Wood's notion of \emph{traceable} in the rank $2$ case to \emph{pseudoregular} in the general case; this notion (but without the name) was also studied in the work of Voight \cite{Voight:quatring}, following Gross--Lucianovic \cite{GL} and Bhargava \cite{Bhargava2}.  


\subsection{Contents}
In \cref{sec:setup}, we give terminology and notation for quadratic modules, framed objects, stacks, and  Clifford algebras. In \cref{sec:pseudo}, we explore pseudoregular modules in the context of the Clifford map.   In \cref{sec:norms}, we define a norm functor proving \Cref{Main Thm GO} and compare to Wood's result. In \cref{sec:dor}, we discuss orientations via the universal Clifford center and prove the corollaries of \Cref{Main Thm GO}.  In \cref{sec: universal composition law}, we develop the universal composition law, including the recovery of classical composition laws.  In \cref{sec:rigid}, we discuss rigidifications and prove the results in \cref{sec:simlattices}.  In \cref{sec: app}, we discuss the applications of the previous results to lattices over Dedekind domains and genera and class sets of primitive binary quadratic modules. In \cref{recentwork}, we review recent work generalizing Gauss composition.  

\subsection{Acknowledgements}

The authors would like to thank Eran Assaf, Asher Auel, ChatGPT 5.5, and Melanie Matchett Wood for helpful conversations.  Voight was supported by the Simons Foundation (550029) and (SFI-MPS-Infra\-structure-00008650). Wu was partially supported by National Science Foundation DMS-2200845 and DMS-2401601.

\section{Background} \label{sec:setup}

We begin with a bit of background and setup.  For further reading, see Voight \cite{Voight:quadmonoid}, Bichsel--Knus \cite{BK}, and Auel \cite{Auel11}.

\subsection{Quadratic modules} \label{sec:quadmodules}

Let $X$ be a scheme. A \defi{quadratic $\scrO_X$-module} is a triple $(\scrM, \scrL, Q)$ where $\scrM$ is a locally free $\scrO_X$-module of finite rank, $\scrL$ is an invertible $\scrO_X$-module, and $Q\colon \scrM\to\scrL$ is a \defi{quadratic map}: for all open sets $U\subseteq X$, we have 
\begin{enumerate}[label={\textup{(\roman*)}}]
    \item $Q(rx)=r^2Q(x)$ for all $r\in\scrO_X(U)$ and $x\in \scrM(U)$;  and
    \item The map $T\colon \scrM(U)\times \scrM(U)\to\scrL(U)$ defined by $T(x,y)=Q(x+y)-Q(x)-Q(y)$ for $x,y\in\scrM(U)$ is $\scrO_X(U)$-bilinear.
\end{enumerate}
We call $T \colon \scrM \times \scrM \to \scrL$ the \defi{associated bilinear form}; we often just write $Q \colon \scrM \to \scrL$ for the triple $(\scrM,\scrL,Q)$ and refer to it as a quadratic module.  As in the introduction, a \defi{quadratic form} is a quadratic module with $\scrL=\scrO_X$.

By $\Sym^k \scrM$ we mean the usual quotient of $\scrM^{\otimes k}$, and by $\Sym_k \scrM$ we mean the submodule of symmetric tensors.  Given $\scrM$ and $\scrL$, a quadratic map $Q\colon\scrM\to\scrL$ is equivalently a global section of the sheaf 
\begin{equation} \label{eqn:sym2sym2}
(\Sym_2 \scrM)^\vee \otimes \scrL \cong \Sym^2 (\scrM^\vee) \otimes \scrL.
\end{equation}
This isomorphism arises via restriction of the natural tensor pairing (see Voight \cite{Voight:quadmonoid}*{section 2}).

A \defi{similarity} between two quadratic modules $Q\colon\scrM\to\scrL$ and $Q'\colon \scrM'\to\scrL'$ is a pair of $\scrO_X$-module isomorphisms $\varphi\colon\scrM\to\scrM'$ and $\lambda\colon\scrL\to\scrL'$ such that $Q'\circ \varphi=\lambda \circ Q$, i.e., the diagram
\begin{equation} \label{eqn:commdiag}
\begin{tikzcd}
	\scrM & \scrL \\
	\scrM' & \scrL'
	\arrow["{Q}", from=1-1, to=1-2]
	\arrow["{\varphi}", swap, from=1-1, to=2-1]
	\arrow["{Q'}", from=2-1, to=2-2]
	\arrow["{\lambda}", from=1-2, to=2-2]
\end{tikzcd}
\end{equation}
commutes.  A \defi{similitude} is a similarity with $\scrL=\scrL'$, and an \defi{isometry} is a similitude with $\lambda=\id$, i.e., a commutative diagram
\begin{equation} \label{eqn:commdiag_isom}
\begin{tikzcd}[
  row sep={0.5ex}
]
	\scrM &  \\
    & \scrL \\
	\scrM' & 
	\arrow["{Q}" pos=0.3, from=1-1, to=2-2]
	\arrow["{\varphi}", swap, from=1-1, to=3-1]
	\arrow["{Q'}" pos=0.3, swap, from=3-1, to=2-2]
\end{tikzcd}
\end{equation}
When $Q=Q'$, then the self-similitudes and self-isometries are represented by closed subgroup schemes of $\Aut(\scrM) \times \Aut(\scrL)$ and $\Aut(\scrM)$, respectively.  

Let $T$ be the associated bilinear form of $Q$.  The determinant of $T$ is the symmetric pairing
\begin{equation}
\det T \colon \tbigwedge^m \scrM \otimes \tbigwedge^m \scrM \to \scrL^{\otimes m}
\end{equation}
given locally by 
\begin{equation} 
(x_1 \wedge \cdots \wedge x_m) \otimes (y_1 \wedge \cdots \wedge y_m) \mapsto \det(T(x_i,y_j))_{i,j}. 
\end{equation}
We define the \defi{(signed) discriminant} of $Q$ to be the quadratic $\scrO_X$-module of rank $1$
\begin{equation} 
\disc Q \colon \tbigwedge^m \scrM \to \scrL^{\otimes m} 
\end{equation}
given by the signed determinant $(-1)^{m(m-1)/2} \det T$.  Equivalently, $\disc Q$ is the corresponding global section of $\scrL^{\otimes m}\otimes(\tbigwedge^m \scrM)^{\otimes{-2}}$.  On an affine open $U=\Spec R$ where $M$ and $L$ are free, this recovers the usual signed determinant of the Gram matrix of $T$, and these local expressions glue.

\begin{example} \label{exm:b24ac}
When $m=2$, twisting by $\scrL^\vee$ gives a quadratic form 
\[ \tbigwedge^2 \scrM \otimes \scrL^\vee \to \scrO_X. \]
If $Q \colon R^2 \to R$ is given by $Q(x,y)=ax^2+bxy+cy^2$ with $a,b,c \in R$, then \begin{equation}
(\disc Q)(e_1 \wedge e_2)=b^2-4ac. 
\end{equation}
\end{example}

\subsection{Algebras}

An \defi{$\scrO_X$-algebra} of rank $n \geq 1$ is an $\scrO_X$-algebra (associative, with $1$, not necessarily commutative) that is locally free of rank $n$ as an $\scrO_X$-module.  For each $n$, let  $\catAlg_n$ be the category of algebras of rank $n$.  

When $n=2$, we use the term \defi{quadratic} $\scrO_X$-algebra.  A quadratic $\scrO_X$-algebra is equivalently given by the pushforward $\pi_*\scrO_Y$ where $\pi\colon Y \to X$ is a locally free morphism of degree $2$.  By a slight abuse of notation, we will write $\scrO_Y$ instead of $\pi_*\scrO_Y$. 

Let $\scrO_Y$ be a quadratic $\scrO_X$-algebra.  Then there is a unique \defi{standard involution} $\overline{\phantom{x}}$ on $\scrO_Y$.  The \defi{discriminant} $\disc \scrO_Y$ is the quadratic map $d \colon \tbigwedge^2 \scrO_Y \to \scrO_X$ corresponding to the linear map $(\tbigwedge^2 \scrO_Y)^{\otimes 2} \to \scrO_X$ defined by 
\[ d((x\wedge y)\otimes(z\wedge w))=(x\overline{y}-\overline{x}y)(z\overline{w}-\overline{z}w)
\] 
for all $x,y,z,w \in \scrO_Y(U)$ with $U \subseteq X$ open \cite{Voight:quadmonoid}*{Proposition 3.13}.  

\begin{example} \label{exm:t24n}
If $Y = \Spec S \to X=\Spec R$ with $S=R[\alpha]=R[x]/\langle x^2-tx+n \rangle$, then $\overline{\alpha}=t-\alpha$ and $(\disc S)(1\wedge \alpha) = t^2-4n$ as usual.
\end{example}

\subsection{Framed atlases and quotient stacks} \label{sec:framed-atlases}

We will see that the categories considered in this paper have the additional structure of a quotient stack coming from a scheme atlas, with a framed universal object.  In this section, we show this for quadratic modules and quadratic algebras.  All stacks are taken over the big Zariski site of schemes, unless otherwise stated; the Zariski topology is enough here because locally free modules and line bundles are Zariski-locally trivial.  

A \defi{frame} for a (free) $\scrO_X$-module $\scrM$ of rank $m \geq 1$ is an $\scrO_X$-basis $(e_1,\dots,e_m)$ with $e_i \in \scrM(X)$, so $\scrM=\scrO_X e_1 \oplus \dots \oplus \scrO_X e_m$.  A \defi{framed quadratic $\scrO_X$-module} is a quadratic module $Q \colon \scrM \to \scrL$ where $\scrM$ and $\scrL$ are framed (of ranks $m$ and $1$, respectively).  A framed similarity of framed quadratic $\scrO_X$-modules is a similarity which preserves the framing.  The automorphism group of a framed quadratic $\scrO_X$-module is therefore trivial.  

Let $\catQMod=\catQMod_m$ be the category of quadratic $\scrO_X$-modules of rank $m \geq 1$ fibered over $\catSch$ the category of schemes.  Let $\Qfr$ be the framed category.  


\begin{lem}\label{lem: universal quadratic form}
Let $R \colonequals \Z[a_{ij}]_{1\leq i\leq j\leq m}$ be a polynomial ring in $m(m+1)/2$ variables over $\Z$.  Then the category $\Qfr$ is represented by $\Spec(R)$, with universal object  
\begin{equation} \label{eqn:univquad}
\begin{aligned}
Q \colon R^m &\to R \\
Q(x_1,\dots,x_m) &= \sum_{1 \leq i \leq j \leq m} a_{ij} x_i x_j.
\end{aligned}
\end{equation}
\end{lem}

\begin{proof}
Over a scheme $X$, a quadratic module $Q \colon \scrM \to \scrL$ with frame $\beta=(e_1,\dots,e_m)$ and $\scrL=\scrO_X f$ is uniquely determined by the values $Q(e_i)=a_{ii}f$ for $i=1,\dots,m$ and $Q(e_i+e_j)-Q(e_i)-Q(e_j)=a_{ij}f$ with $1 \leq i < j \leq m$, and conversely.
\end{proof}

A change of module basis is given by a left action of $\GL_m$, inducing a right action of $\GL_m$ on the quadratic map, with the rescaling action, i.e., 
\begin{equation} \label{eqn:AUscal}
(Q | (A,u))(x) = uQ(Ax)
\end{equation}
for $A \in \GL_m(R)$ and $u \in R^\times$ for $Q \colon M \to L$ over $R$.  

We refer to Stacks \cite{stacks-project}*{Tags 04UI, 04UV} for background on quotient stacks.  

\begin{cor}
The category $\catQMod_m \simeq [\mathbb{A}^{m(m+1)/2}/(\GL_m \times \Gm)]$ is a stack quotient.
\end{cor}

\begin{proof}
An object of the quotient stack over a test scheme $T$ is an
$(\GL_m\times \Gm)_T$-torsor $P \to T$, together with an
equivariant map $P \to \Spec R$.  The $\GL_m$-part of the torsor is the
frame torsor of a locally free rank $m$ bundle $\scrM$, while the
$\Gm$-part is the frame torsor of a line bundle $\scrL$.  The equivariant map $P \to \Spec R$ is precisely the Zariski descent datum for the universal quadratic form, and hence gives a quadratic map.  Conversely, on a Zariski open cover of $T$, the bundles $\scrM$ and $\scrL$ are trivial, and by \Cref{lem: universal quadratic form}, $Q$ is given by $m(m+1)/2$ functions hence by a
map to $Z$.  On overlaps, changing the basis of $\scrM$ and the
trivialization of $\scrL$ changes these coefficients by the above
$\GL_m\times\Gm$-action.  Thus the local maps glue exactly
to an object of the quotient stack.  The construction is functorial in
$T$, and gives the claimed equivalence.
\end{proof}

In particular, we are interested in the case $m=2$ of binary quadratic modules. A free binary quadratic module $Q \colon \calO_X^2 \to \calO_X$ is defined by $Q(xe_1+ye_2)=Q(x,y)=ax^2+bxy+cy^2$ with $a,b,c \in \calO_X(X)$, and we abbreviate $Q=[a,b,c]$.  In particular, framed binary quadratic modules are parametrized by 
\begin{equation} \label{eqn:Qfr}
\Qfr \simeq \mathbb{A}^3_{\mathbb{Z}}=\Spec \Z[a,b,c]
\end{equation}
in the familiar way, and 
\begin{equation} \label{catbqmodeq}
\catBQMod \simeq [\mathbb{A}^3/(\GL_2 \times \Gm)]
\end{equation}
where $\GL_2$ acts on $\mathbb{A}^3$ in the symmetric square representation and $\Gm$ acts by rescaling the value frame.

Similarly, a \defi{(unital) frame} for an $\scrO_X$-algebra of finite rank is a frame $\beta$ for the underlying $\scrO_X$-module such that $\beta_1=1$.  

Let $\catAlg_m$ denote the stack whose objects over a scheme $X$ are $\scrO_X$-algebras of rank $m \geq 1$ and whose morphisms are $\scrO_X$-algebra isomorphisms.  Let $\catAlg_m^{\textup{fr}}$ be the framed category.  Then $\catAlg_m^{\textup{fr}}$ is the closed subscheme of the affine space of multiplication tables defined by identity and associativity conditions, and $\catAlg_m$ is the quotient stack by the algebraic subgroup of $\GL_m$ which maps the first basis element to itself.  

Again we are primarily interested in the case $m=2$.  The category $\catAlg_2^{\textup{fr}}$ of framed quadratic $\scrO_X$-algebras under morphisms preserving the framing is represented by the affine scheme 
\begin{equation} \label{eqn:quadalguniv}
\catAlg_2^{\textup{fr}} \simeq \Spec \Z[t,n],
\end{equation}
with universal framed quadratic algebra $\Z[t,n][x]/\langle x^2-tx+n\rangle$ with frame $(1,x)$.  A change of frame is of the form $x \mapsto u(x+r)$ with $u \in \scrO_X(X)^\times$ and $r \in \scrO_X(X)$, corresponding to the coordinate change 
\begin{equation}  \label{eqn:coordchange}
(t,n) \mapsto (u(t+2r),u^2(n+tr+r^2)).
\end{equation}
Thus $\catQuad$ can be identified with the stack quotient
\begin{equation} 
\catQuad \simeq [(\Spec \Z[t,n])/(\G_a \rtimes \Gm)].
\end{equation}

We work in the 2-category of stacks over $\catQuad$: a morphism over $\catQuad$ is equipped with a specified 2-isomorphism between the two composites to $\catQuad$.  We suppress this 2-isomorphism when no confusion can result.

We also use the Picard stack $\catPic\simeq B\G_m$ of invertible sheaves, together with the morphism $\catQMod_m\to\catPic$ defined by taking the value bundle.

\subsection{Fibers of morphisms of stacks} \label{sec:fibers}

A key observation that goes a long way to explaining our main results involves fibers of a morphism of stacks.  We briefly review these notions: for more, see the Stacks Project \cite{stacks-project}*{Tags 003O, 0040, 02ZL}.

Let $f \colon \mathcal{X} \to \mathcal{Z}$ be a morphism of stacks.  Given a test scheme $T$ and an object $Z \in \mathcal{Z}(T)$, the \defi{$2$-fiber of $f$ over $Z$} is the stack 
\begin{equation}
\mathcal{X}_Z \colonequals T \times_{\mathcal{Z}} \mathcal{X}
\end{equation}
where the fiber product is the $2$-fiber product.  More precisely, for $T$-scheme $T' \to T$, the groupoid $\mathcal{X}_Z(T')$ has:
\begin{itemize}
\item objects $(X,\iota)$ where $X \in \mathcal{X}(T')$ and $\iota \colon f(X) \xrightarrow{\sim} Z_{T'} = Z \times_{T} T'$ is an isomorphism in $\mathcal{Z}(T')$; and
\item morphisms $(X_1,\iota_1) \to (X_2,\iota_2)$ are isomorphisms $\beta \colon X_1 \to X_2$ in $\mathcal{X}(T')$ such that the diagram 
\begin{equation}
\begin{tikzcd}
	f(X_1) & f(X_2) \\
    Z_{T'} & Z_{T'}
	\arrow["{f(\beta)}", from=1-1, to=1-2]
	\arrow["\iota_1", from=1-1, to=2-1]
	\arrow["\iota_2", from=1-2, to=2-2]
    \arrow[r, equal, from=2-1, to=2-2]
\end{tikzcd}
\end{equation}
commutes, i.e., $\iota_2 \circ f(\beta) = \iota_1$.  
\end{itemize}

The usual fiber of a map of sets is the collection of $x$ such that $f(x)=z$ for $z \in Z$; the $2$-fiber is the groupoid formed by $X$ equipped with a \emph{specified} isomorphism $\iota \colon f(X) \xrightarrow{\sim} Z$.  

The following lemma is key to our categorical point of view, so risking overkill we give a careful complete argument.

\begin{lem} \label{lem:fibres-of-equivalence}
Let $p_{\mathcal X} \colon \mathcal X\to \mathcal Z$ and $p_{\mathcal Y}\colon \mathcal Y\to \mathcal Z$ be stacks fibered in groupoids over a fixed site, and let $f \colon \mathcal X \to \mathcal Y$ be an equivalence over $\mathcal{Z}$.  Then for every test scheme $T$ and every object
$Z\in\mathcal Z(T)$, the induced functor on $2$-fibers
\[
f_Z\colon \mathcal X_Z\to \mathcal Y_Z
\]
is an equivalence of categories.

Moreover, for every morphism $\alpha\colon Z\to Z'$ in
$\mathcal Z(T)$, there are canonical pullback functors $\alpha_{\mathcal X}^*\colon \mathcal X_{Z'}\to \mathcal X_Z$ and $\alpha_{\mathcal Y}^*\colon \mathcal Y_{Z'}\to \mathcal Y_Z$, 
and the square
\begin{equation}
\begin{tikzcd}
\mathcal X_{Z'} \arrow[r,"f_{Z'}"] \arrow[d,"\alpha_{\mathcal X}^*"']
&
\mathcal Y_{Z'} \arrow[d,"\alpha_{\mathcal Y}^*"]
\\
\mathcal X_Z \arrow[r,"f_Z"']
&
\mathcal Y_Z
\end{tikzcd}
\end{equation}
commutes up to a canonical isomorphism.  If $\alpha$ is an
isomorphism, then $\alpha_{\mathcal X}^*$ and
$\alpha_{\mathcal Y}^*$ are equivalences.  
\end{lem}

\begin{proof}
For notational simplicity, we suppose that the triangle over $\mathcal Z$ strictly commutes, so that
$p_{\mathcal Y}f=p_{\mathcal X}$.  The general case is identical after composing the displayed rigidifying isomorphisms with the chosen $2$-isomorphism $p_{\mathcal Y}f\simeq p_{\mathcal X}$.  

We define a functor
\begin{equation}
f_Z\colon \mathcal X_Z\to\mathcal Y_Z
\end{equation}
by $f_Z(X,\iota) = (f(X),\iota)$ on objects and applying $f$ on morphisms.  Since $f\colon\mathcal X\to\mathcal Y$ is an equivalence over
$\mathcal Z$, its base change along $Z\colon T\to\mathcal Z$ is again
an equivalence.  Formally, this follows from
\cite{stacks-project}*{Tag 02XB} with $L=f$, $K=\mathrm{id}_T$, and
$M=\mathrm{id}_{\mathcal Z}$.  Equivalently, one may check on fibers
over $T'\to T$, using the fiberwise criterion for equivalences of
categories fibered in groupoids \cite{stacks-project}*{Tag 003Z}.  

Now let $\alpha\colon Z\to Z'$ be a morphism in $\mathcal Z(T)$.
For $U\to T$, write $\alpha_U\colon Z|_U\to Z'|_U$ for its pullback to $U$.  Define
\begin{equation}
\alpha_{\mathcal X}^*\colon \mathcal X_{Z'}\to\mathcal X_Z
\end{equation}
by $\alpha_{\mathcal X}^*(X,\iota')=(X,\alpha_U^{-1} \circ \iota')$.  The same construction gives $\alpha_{\mathcal Y}^*\colon \mathcal Y_{Z'}\to\mathcal Y_Z$.  These assignments are functorial on morphisms, again by the compatibility condition in the explicit description of the $2$-fiber product.

Compatibility with $f$ is immediate from the definitions: in the
strictly commuting case,
\begin{equation}
f_Z\circ \alpha_{\mathcal X}^*
=
\alpha_{\mathcal Y}^*\circ f_{Z'}.  
\end{equation}

Finally, if $\alpha$ is an isomorphism, then $(\alpha^{-1})^*$ is a
quasi-inverse to $\alpha^*$.
\end{proof}

\subsection{Clifford functor}

For an invertible $\scrO_X$-module $\scrL$, we define $\scrL^\vee\colonequals \HHom_{\scrO_X}(\scrL,\scrO_X)$. 

Let $Q \colon \scrM \to \scrL$ be a quadratic $\scrO_X$-module.  We define the \defi{even tensor $\scrO_X$-algebra} (associated to $\scrM,\scrL$) by 
\begin{equation}
\TTen^0(\scrM;\scrL)\colonequals\displaystyle\bigoplus_{d=0}^{\infty}\scrM^{\otimes 2d} \otimes (\scrL^\vee)^{\otimes d}
\end{equation}
(with the canonical reordering to get a multiplication).  

Let $\scrJ^0(Q)$ be the quasi-coherent two-sided ideal sheaf of $\TTen^0(\scrM;\scrL)$ generated by the sections defined over $U \subseteq X$ open by
\begin{equation}
\begin{aligned}
&\langle x\otimes x\otimes f-f(Q(x)) : x \in \scrM(U), f \in \scrL^\vee(U) \rangle \\
&\quad + \langle x\otimes y\otimes y \otimes z\otimes f\otimes g-f(Q(y))\, x\otimes z\otimes g : x,y,z \in \scrM(U), f,g \in \scrL^\vee(U) \rangle.
\end{aligned}
\end{equation}
We define the \defi{even Clifford algebra} of $Q$ to be the quotient 
\begin{equation}
\Clf^0(Q)\colonequals\TTen^0(\scrM;\scrL)/\scrJ^0(Q).
\end{equation}

\begin{prop}  \label{prop:evencliff}
The even Clifford map defines a morphism of stacks $\catQMod_m \to \catAlg_{2^{m-1}}$, that is, between
\begin{center}
quadratic modules of rank $m$ under similarities
\end{center}
and
\begin{center}
algebras of rank $2^{m-1}$ under algebra isomorphisms.
\end{center}
\end{prop}

\begin{proof}
See Auel \cite{Auel15}*{Proposition 1.2} and for a slightly different but equivalent construction using the Rees algebra, see Bichsel--Knus \cite{BK}*{\S 3}.
\end{proof}

\begin{example}\label{binary even cliff mod O_X}
    In the binary case, we have a canonical isomorphism
    \[ \Clf^0(Q)/\calO_X \cong \tbigwedge^2 \scrM \otimes \scrL^{\vee}\]
    as $\calO_X$-modules.
    
    If $X=\Spec R$ is affine with $S=\Clf^0(Q)$, then we get a noncanonical splitting
    \[ S \simeq R \oplus S/R \simeq R \oplus (\tbigwedge^2 M \otimes L^\vee). \]
    See Voight \cite{Voight:quadmonoid}*{Lemma 3.2, Remark 3.3} for further freeness comments.
\end{example}

Similarly, define the \defi{odd tensor $\scrO_X$-bimodule} by
\begin{equation}
\TTen^1(\scrM;\scrL)\colonequals\displaystyle\bigoplus_{\substack{d=1\\d \text{ odd}}}^\infty\scrM^{\otimes d}\otimes (\scrL^{\vee})^{\otimes \lfloor d/2 \rfloor}.
\end{equation}
Then $\TTen^1(\scrM;\scrL)$ is a graded $\TTen^0(\scrM;\scrL)$-bimodule under the tensor multiplication law. Let $\scrJ^1(Q)$ be the $\TTen^0(\scrM;\scrL)$-sub-bimodule of $\TTen^1(\scrM;\scrL)$ defined over $U\subseteq X$ open by 
\begin{equation}
    \begin{aligned}
     \langle x\otimes x \otimes y\otimes f-f(Q(x))y, y\otimes x\otimes x\otimes f-f(Q(x))y: x,y\in \scrM(U), f\in\scrL^{\vee}(U)\rangle.
    \end{aligned}
\end{equation}
 We define the \defi{odd Clifford bimodule} of $Q$ to be the quotient
\begin{equation}
 \Clf^1(Q)\colonequals\TTen^1(\scrM;\scrL)/\scrJ^1(Q).
\end{equation}
Then under the natural multiplication, $\Clf^1(Q)$ has the structure of a $\Clf^0(Q)$-bimodule; it is locally free of rank $2^{m-1}$ as an $\scrO_X$-module, and its formation commutes with arbitrary base change.

\begin{exm} \label{exm:a3bdiso}
In the framed binary case $Q=[a,b,c]$ over $R$, with value frame fixed and $\alpha \colonequals e_1e_2 \otimes f^\vee$.  Then
\begin{equation}
\Clf^0(Q)\simeq R[\alpha]=R[x]/\langle x^2-bx+ac \rangle 
\end{equation}
Thus on framed atlases, the map $\Clf^0$ is
\begin{equation}
\begin{aligned}
\mathbb A^3_{\mathbb Z} &\to \mathbb A^2_{\mathbb Z} \\
(a,b,c) &\mapsto (b,ac).
\end{aligned}
\end{equation}
The discriminant $b^2-4ac$ is preserved.
\end{exm}

\section{Pseudoregularity} \label{sec:pseudo}

In this section, we explore modules which from the point of view of characteristic polynomials look like the regular module (affording the regular representation).  We then show this holds for the odd Clifford bimodule over the even Clifford algebra.  

\subsection{Pseudoregular modules}

We first work locally.  Let $R$ be a commutative ring (with $1$).  Let $A$ be an $R$-algebra (with $1$, associative, not necessarily commutative) that is free of rank $n \geq 1$ as an $R$-module.  Left multiplication defines an $R$-algebra homomorphism
$\lambda_A \colon A \to \End_R(A) \simeq \M_n(R)$, called the \defi{(left) regular representation} of $A$ (over $R$).  We obtain for each $\alpha \in A$ a (left) characteristic polynomial $c_A(\alpha;T)=\det(T-\lambda_A(\alpha)) \in R[T]$.  

Let $I$ be a left $A$-module.  We say that $I$ is \defi{$R$-rank balanced} if $I$ is also free of rank $n$ as an $R$-module.  

Suppose that $I$ is $R$-rank balanced.  Then again the left $A$-module structure on $I$ yields an $R$-algebra homomorphism $\lambda_I \colon A \to \End_R(I) \simeq \M_n(R)$ and $c_I(\alpha;T) \in R[T]$.  

\begin{definition} \label{defn:pseudoregular_affine}
We say $I$ is $(R,A)$-\defi{pseudoregular} if $I$ is $R$-rank balanced and for all $\alpha \in A$ we have $c_A(\alpha;T)=c_I(\alpha;T)$.  

We say $I$ is \defi{universally $(R,A)$-pseudoregular} if $I \otimes_R R'$ is $(R',A \otimes_R R')$-pseudoregular for all ring homomorphisms $R \to R'$.  
\end{definition} 

The name is justified as $I$ looks like the regular representation of $A$ from the point of view of characteristic polynomials. When clear from context, we will often drop the prefix $(R,A)$ and say simply that $I$ is pseudoregular.

\begin{remark}
One could extend \Cref{defn:pseudoregular_affine} to the case where $M$ is a left $A$-module that is a free $R$-module of rank $nd$ with $d \geq 1$ and ask that $c_M(x;T)=c_A(x;T)^d$.  
\end{remark}

\begin{example}
$A$ is regular hence pseudoregular as an $A$-module.
\end{example}

We record the following useful criterion. 

\begin{lem} \label{lem:pseudoregular_universal}
Let $e_1,\dots,e_m$ be an $R$-basis for $A$ and let $I$ be an $R$-rank balanced left $A$-module.  Let 
\begin{equation}
  R[x] \colonequals R[x_1,\dots,x_m]
\end{equation}
and let $\xi \colonequals x_1e_1+\dots+x_me_m \in A_{R[x]} \colonequals A \otimes_R R[x]$.  Then $I$ is $(R,A)$-universally pseudoregular if and only if 
\begin{equation} \label{eqn:capxi}
  c_{A_{R[x]}}(\xi;T)=c_{I_{R[x]}}(\xi;T) \in R[x][T].
\end{equation}
\end{lem}

We call the element $\xi$ the \defi{universal element} of $A$ (with respect to the given basis).

\begin{proof}
If $I$ is universally pseudoregular, then \eqref{eqn:capxi} holds by definition.  Conversely, suppose that \eqref{eqn:capxi} holds.  Let $R \to R'$ be a ring homomorphism, and let
$\alpha \in A_{R'}$.  Then expressing $\alpha$ in the basis $e_1,\dots,e_m$ for $A_{R'}$ we see it is the specialization of $\xi$ under a unique ring homomorphism $R[x] \to R'$.  Specializing \eqref{eqn:capxi} shows that $c_{A_{R'}}(\alpha;T) = c_{I_{R'}}(\alpha;T) \in R'[T]$.  Since $\alpha$ was arbitrary, $I_{R'}$ is pseudoregular; since $R \to R'$ was arbitrary, $I$ is universally pseudoregular.
\end{proof}

In fact, by a formal calculation universal pseudoregularity is automatic.  

\begin{prop}
$I$ pseudoregular if and only if it is universally pseudoregular.
\end{prop}

\begin{proof}
This follows naturally from results on homogeneous multiplicative polynomial laws \cite{Rob63}.  Suppose that $I$ is pseudoregular.  We recall a consequence of Amitsur's formula for the determinant of the sum of two matrices: if two finite
free representations of an $R$-algebra have the same characteristic
polynomial on an $R$-module generating set of the algebra, then they
have the same characteristic polynomial on every element \cite{Ami80}*{Theorems~A--B}.  See also Chenevier \cite{Che14}*{Lemma~1.12(ii), formula~(1.5)} and
Reutenauer--Sch{\"u}tzenberger \cite{RS87} for some explicit formulas in this case.

We apply this to the universal element, as in \Cref{lem:pseudoregular_universal}.  Let $R[x]=R[x_1,\ldots,x_n]$.  The elements $e_1,\ldots,e_n$ are an $R[x]$-module basis of $A_{R[x]}$.  By the previous paragraph, we get 
\[
c_{A_{R[x]}}(\xi;T)=c_{I_{R[x]}}(\xi;T).
\]
The result follows from \Cref{lem:pseudoregular_universal}.

The converse is immediate.  
\end{proof}

\begin{lem} \label{lem:Ifreepseudo}
The following statements hold.
\begin{enumalph}
\item If $I$ is free of rank $1$ as an $A$-module, then $I$ is universally $(R,A)$-pseudoregular.  
\item If $R \to R'$ is an injective ring homomorphism and $I \otimes_{R} R'$ is (universally) $(R',A \otimes_{R} R')$-pseudoregular, then $I$ is (universally) $(R,A)$-pseudoregular.  
\item If $R$ is a domain with field of fractions $F$, and $A_F$ is a division algebra, then $I$ is universally $(R,A)$-pseudoregular.
\end{enumalph}
\end{lem}

\begin{proof}
For part (a), writing $I=Ae$, an $R$-basis $\beta$ for $A$ yields $\beta e$ an $R$-basis for $I$, and accordingly $[\lambda_A(\alpha)]_\beta=[\lambda_I(\alpha)]_{\beta e}$ so in particular $c_A(\alpha;T)=c_I(\alpha;T)$.  Repeating this over an arbitrary base gives the conclusion.  

For (b), let $I' \colonequals I\otimes_{R} R'$ and $A' \colonequals A\otimes_{R} R'$.  We identify $R$ with its image $R \to R'$.  For $\alpha \in A$, since an $R$-basis for $A$ yields an $R'$-basis for $A_{R'}$, we have $c_{I'}(\alpha;T)=c_{I}(\alpha;T) \in R[T] \subseteq R'[T]$.  Similarly, we have $c_{A'}(\alpha;T)=c_A(\alpha;T) \in R'[T]$. By pseudoregularity we have $c_{I'}(\alpha;T)=c_{A'}(\alpha;T)$, so $c_I(\alpha;T)=c_A(\alpha;T) \in R'[T]$.  Thus $c_I(\alpha;T)=c_A(\alpha;T) \in R[T]$.  We conclude that $I$ is $(R,A)$-pseudoregular.  

For the universal statement, we consider the universal element $\xi \in A_{R[x]}$.  Then under the inclusion $A_{R[x]} \hookrightarrow A_{R'[x]}$, in fact $\xi$ is also the universal element for $A_{R'[x]}$.  By \Cref{lem:pseudoregular_universal} we get $c_{A_{R'[x]}}(\xi;T)=c_{I_{R'[x]}}(\xi;T) \in R'[x][T]$.  As in the previous statement, this equality holds already in $R[x][T]$, which again by \Cref{lem:pseudoregular_universal} implies that $I$ is universally $(R,A)$-pseudoregular.  

For (c), $I_F$ is a rank-balanced module over $A_F$ so $1$-dimensional as a left $A_F$-vector space.  So by part (a), $I_F$ is universally $(F,A_F)$-pseudoregular.  By (b), since $R \hookrightarrow F$ we conclude that $I$ is universally $(R,A)$-pseudoregular.  
\end{proof}

\begin{exm} \label{exm:apeislon}
Let $A \colonequals R[\varepsilon]$ where $\varepsilon^2=0$ and let $I \colonequals R^2$ where $\varepsilon$ acts by $0$ on $I$.  Then $[\varepsilon]_A=\begin{pmatrix} 0 & 0 \\ 1 & 0 \end{pmatrix}$ and $[\varepsilon]_I = 0$, but $c_A(\varepsilon;T)=c_I(\varepsilon;T)=T^2$, so $I$ is pseudoregular.
\end{exm}

\begin{exm} \label{exm:faith}
Let $A \colonequals R \times R$, let $\frakp=R \times \{0\}$ and let $I \colonequals A/\frakp \oplus A/\frakp \simeq R \oplus R$.  Then $\rk_R A = \rk_R I = 2$, but $I$ is not pseudoregular because $\alpha=(a,b)$ has $c_A(\alpha;T)=(T-a)(T-b)$, but $c_I(\alpha;T)=(T-b)^2$.  Indeed, $I$ is not faithful as an $A$-module, as $(1,0)$ annihilates $I$.
\end{exm}


\begin{exm}
    Let $A\subset \M_2(R)$ be the $R$-subalgebra of upper triangular matrices.  Then column vectors $R^2$ with the natural action of left multiplication is a faithful $A$-module.  But $A$ also admits the natural $R$-algebra quotient $A \to R \times R$ by $\begin{pmatrix} a & b \\ 0 & d \end{pmatrix} \mapsto (a,d)$.  Let $I'$ be the further quotient to the second factor and $I \colonequals R^2 \oplus I'$.  Then $c_A(\alpha;T)=\det(T-\alpha)(T-a)=(T-a)^2(T-d)$ but $c_I(\alpha;T)=\det(T-\alpha)(T-d)=(T-a)(T-d)^2$.  Since $A,I \simeq R^3$ as $R$-modules, $I$ is faithful but not pseudoregular as an $A$-module.  
\end{exm}

We now define the global notion, extending Zariski locally.  First, if $A$ is a locally free $R$-algebra of rank $n$ and $I$ is an $A$-module which is locally free of rank $n$ as an $R$-module, then we again have (uniquely defined) characteristic polynomials, so we just repeat \Cref{defn:pseudoregular_affine}.  

Finally, let $X$ be a scheme, let $\scrA$ be a locally free $\calO_X$-algebra.  Let $\scrI$ be an $\scrO_X$-rank balanced $\scrA$-module, i.e., locally free of rank $n$ as an $\scrO_X$-module.

\begin{definition}
We say $\scrI$ is (universally) $(\scrO_X,\scrA)$-\defi{pseudoregular} if there exists an affine open cover $X=\bigcup_i U_i$ such that $\scrI(U_i)$ is (universally) $(\scrO_X(U_i),\scrA(U_i))$-pseudoregular for all $i$.  
\end{definition}

It is straightforward to check that $\scrI$ is (universally) pseudoregular if and only if for all affine open $U \subseteq X$ we have $\scrI(U)$ is (universally) pseudoregular. 

We make similar definitions for right modules and for bimodules (as both a left and right module).

\begin{definition} \label{def:pseudoregcat}
Let $n \in \Z_{\geq 1}$.  We define $\catPSReg_n$, the category of pseudoregular bimodules of rank $n$ over algebras of rank $n$, fibered in groupoids over the category of schemes, as follows:
\begin{itemize}
\item an object is a triple $(X,\scrA,\scrI)$ where $\scrA$ is an $\scrO_X$-algebra of rank $n$ and $\scrI$ is a universally pseudoregular $\scrA$-bimodule;  
\item a morphism $(X,\scrA,\scrI) \to (X',\scrA',\scrI')$ is a pair $(f,\sigma,\psi)$ where $f \colon X\to X'$ is a morphism, $\sigma \colon \scrA \xrightarrow{\sim} f^*\scrA'$ is an $\scrO_X$-algebra isomorphism and $\psi \colon \scrI \xrightarrow{\sim} f^*\scrI'$ is a $\sigma$-semilinear $\scrO_X$-bimodule isomorphism, i.e., 
\[ \psi(\alpha x \beta) = \sigma(\alpha) \psi(x)\sigma(\beta) \]
for all $\alpha,\beta \in \scrA(U)$ and $x \in \scrI(U)$ over $U \subseteq X$.  
\end{itemize}
\end{definition}




\begin{remark}
The possible failure of pseudoregularity to be preserved under pullback shows that it is the notion of universal pseudoregularity which will be preferable in stacky formulations.
\end{remark}

With definitions out of the way, we now prove the main results of this section.

\begin{prop} \label{prop:cliffpseudo}
$\Clf^1(Q)$ is universally pseudoregular as a $\Clf^0(Q)$-bimodule.  
\end{prop}

\begin{proof}
We first work with the universal framed quadratic module $Q$ \eqref{eqn:univquad}.  Let $A \colonequals \Clf^0(Q)$ and $I \colonequals \Clf^1(Q)$ be the resulting (universal) framed even Clifford algebra and odd Clifford bimodule.  
Since $R$ is a domain (as a polynomial ring over $\Z$), the map $R \to F \colonequals \Frac R$ is injective.  But $I_{F}$ is free over $A_F$ with basis $e_1$ since $e_1^2=a_{11} \in F^\times$, so by \Cref{lem:Ifreepseudo}(a), $I_{F}$ is universally pseudoregular.  By \Cref{lem:Ifreepseudo}(b), $I$ is universally pseudoregular.  A similar argument holds on the right.

The general case follows by specialization.  As the notion is local, we reduce to the case of an affine cover where the objects are free.  So let $Q \colon M \to L$ be a quadratic module with $M = R^m$ and $L=R$.  Let $A \colonequals \Clf^0(Q)$ and $I \colonequals \Clf^1(Q)$.  Upon a choice of basis, $Q$ arises from the universal framed quadratic module and by functoriality $A$ and $I$ arise from base change of the universal framed even Clifford algebra and odd Clifford bimodule.  We just showed that these are universally pseudoregular, so $A$ and $I$ are pseudoregular.  Repeating over an arbitrary base change (or working with a universal element), we conclude that $A$ and $I$ are in fact universally pseudoregular.  
\end{proof}

\begin{prop} \label{Clif functor similarity}
The association of the odd Clifford bimodule over the even Clifford algebra defines a functor $\catQMod_m \to \catPSReg_{2^{m-1}}$ fibered over $\catAlg_{2^{m-1}}$, i.e., between 
\begin{center}
quadratic modules of rank $m$ under similarities
\end{center}
and
\begin{center}
universally pseudoregular bimodules over algebras of rank $2^{m-1}$ \\ under semilinear bimodule isomorphisms.
\end{center}
\end{prop}

\begin{proof}
For the statement without pseudoregularity, see \cite{Auel15}*{Proposition 1.5} together with \Cref{prop:cliffpseudo}.  The functorial properties in \cite{Auel11}*{Section 4} show that similarities associate to semilinear module isomorphisms between pairs $(\scrA, \scrI)$.
\end{proof}

\subsection{Quadratic pseudoregular modules}

We now specialize to the case of rank $m=2$, first exhibiting a universal object.
An $\scrO_X$-rank balanced $\scrO_Y$-module $\scrI$ is \defi{framed} if both $\scrO_Y$ and $\scrI$ are framed.

\begin{example} \label{exm:goodcliff}
For the framed binary quadratic form $[a,b,c]=ax^2+bxy+cy^2$ over $\scrO_X$, we have $\Clf^0(Q)=\scrO_X \oplus \scrO_X e_1 e_2$ and $\Clf^1(Q)=\scrO_X e_1\oplus \scrO_X e_2$.  We have $(e_1e_2)e_1 = e_1(b-e_1e_2)=be_1 - ae_2$ and $(e_1e_2)e_2=ce_1$ so $[e_1e_2]_{(e_1,e_2)}=\begin{pmatrix} b & c \\ -a & 0 \end{pmatrix}$ (left action on columns).
\end{example}

Let $\scrI$ be a framed quadratic $\scrO_Y$ rank-balanced module, with $1,\alpha$ an $\scrO_X$-basis for $\scrO_Y$ and $e_1,e_2$ an $\scrO_X$-basis for $\scrI$.  With \Cref{exm:goodcliff} in view, we say that $\alpha, e_1,e_2$ is a \defi{good frame} for $\scrO_Y$ and $\scrI$ if $\alpha e_2 \in \scrO_X e_1$. 

\begin{example}
For $S=R[\alpha] \simeq R[x]/\langle x^2-bx+ac \rangle$, the $S$-module $S=I$ has good frame $1,\alpha$ (for $S$) and $1,b-\alpha$ (for $I$), noting $\overline{\alpha}=b-\alpha$.  Indeed, then $[\alpha]_I = \begin{pmatrix} b & ac \\ -1 & 0 \end{pmatrix}$.  
\end{example}

\begin{lem} \label{lem:goodclifflem}
Every framed rank-balanced module over a free quadratic algebra admits a good frame.  
\end{lem}

\begin{proof}
Following Wood \cite{Wood}*{proof of Theorem 2.1}, if $\alpha e_2 = xe_1 + ye_2$ then replacing $\alpha \leftarrow \alpha - y$ we obtain a good frame.
\end{proof}

\begin{remark}
See also Gross--Lucianovic \cite{GL} and Voight \cite{Voight:quatring} for similar notions of \emph{good basis} in the ternary case.
\end{remark}

\begin{lem}\label{lem:autgoodframe}
Over the good-framed parameter scheme, the stabilizer of the good-frame condition inside $\GL_2\times(\G_a\rtimes\G_m)$ is isomorphic to $\GL_2\times\G_m$; the projection to $(\varphi,u)$ is an isomorphism, with $r$ uniquely determined by $(\varphi,u)$.  
\end{lem}

\begin{proof}
Suppose that $1, \alpha', e_1', e_2'$ is a frame for $S$ and $I$. Then $e_1'=pe_1+qe_2$ and $e_2'=se_1+te_2$ for some $\varphi=\begin{pmatrix}
    p & s\\
    q & t
\end{pmatrix}\in \GL_2(R)$ with $d=pt-qs\in R^\times$ and $\alpha'=u(\alpha+r)$ for some $u\in R^\times$ and $r\in R$. We have 
\begin{equation} \label{eqn:alphae2}
\alpha'e_2'=u(\alpha+r)(se_1+te_2)=u(sb+ct+rs)e_1+u(-as+rt)e_2.
\end{equation}
Thus the frame is good if and only if $\alpha'e_2' \in Re_1' = R(pe_1+qe_2)$ if and only if $u(sb+ct+rs)=p\lambda$ and $u(-as+rt)=q\lambda$ for some $\lambda\in R$.  Eliminating $\lambda$ gives 
\begin{equation} 
\begin{aligned}
p(-as+rt) &= (sb+ct+rs)q \\
r(pt-sq) = rd &= asp+sbq+ctq.
\end{aligned}
\end{equation}
Therefore, $r=d^{-1}(asp+sbq+ctq)$ is uniquely determined by $\varphi$ and $u$.  Conversely, assume $r$ is given by this expression. Substituting back into \eqref{eqn:alphae2} gives $pz=qw$ with $z\colonequals -as+rt$ and $w \colonequals sb+ct+rs$. Then we have $ptz=(d+sq)z=tqw$, and thus $z=qd^{-1}(tw-sz)$. Similarly, $w=pd^{-1}(tw-sz)$. Taking $\lambda=ud^{-1}(tw-sz)$ and running the argument in reverse, we have a change of good frame.
\end{proof}

\begin{cor}\label{good frame independence}
 Let $S$ be a free quadratic $R$-algebra and $I$ be a pseudoregular $S$-module with a good frame $\alpha, e_1, e_2$.  Let $[\alpha]_{(e_1,e_2)}=\begin{pmatrix}
     b & c \\
     -a & 0
 \end{pmatrix}$.  
Then the ideal $\langle a,b,c\rangle\subseteq R$ is independent of the choice of good frame. 
\end{cor}

\begin{proof}
Picking up the computation in the previous lemma, we have 
\begin{equation}  \label{eqn:solveforr}
r=d^{-1}(asp+sbq+ctq)\in \langle a, b,c\rangle.
\end{equation}

Now $\alpha'=u(\alpha+r)$ gives 
\begin{equation} \label{eqn:uraction}
[\alpha']_{(e_1,e_2)}=u\begin{pmatrix}
    b+r & c\\
    -a & r
\end{pmatrix}
\end{equation}
where each entry is in $\langle a,b,c\rangle$. Since $\varphi\in \GL_2(R)$, each entry in $[\alpha']_{(e_1', e_2')}=\varphi^{-1}[\alpha']_{(e_1, e_2)}\varphi$  is still in $\langle a,b,c\rangle$. Applying the same argument with the two good frames interchanged gives the reverse containment. 
\end{proof}

The ideals $\langle a,b,c\rangle$ obtained on good-framed opens are compatible under restriction and independent of the chosen good frame by \Cref{good frame independence}; hence they glue to an ideal sheaf $\scrC\subseteq\scrO_X$, called the \defi{content} of $\scrI$.

Indeed, the quadratic case is very special, as the following lemma indicates.  

\begin{prop} \label{prop:traceable}
Let $\scrO_Y$ be a quadratic $\scrO_X$-algebra and let $\scrI$ be a rank-balanced $\scrO_Y$-module.  Then $\scrI$ is pseudoregular if and only if for all open $U \subseteq X$ and all $\alpha \in \scrO_Y(U)$, we have $\Tr_{\scrO_Y(U)}(\alpha) = \Tr_{\scrI(U)}(\alpha)$.
\end{prop}

The condition in \Cref{prop:traceable} is called \defi{traceable} by Wood \cite{Wood}.

\begin{proof}
The assertion is local on $X$, so we may suppose that $S$ and $I$ are free $R$-modules of rank $2$.

If $I$ is pseudoregular, then comparing the coefficient of $T$ in
$c_S(\alpha;T)=c_I(\alpha;T)$ gives $\Tr_S(\alpha)=\Tr_I(\alpha)$ for all $\alpha\in S$.

For the converse, suppose that $\Tr_S(\alpha)=\Tr_I(\alpha)$ for all $\alpha\in S$.  
By \Cref{lem:goodclifflem}, choose a good frame $\gamma,e_1,e_2$.  As in \Cref{good frame independence}, we may write
\[
[\gamma]_{(e_1,e_2)}
=
\begin{pmatrix}
b & c\\
-a & 0
\end{pmatrix},
\]
so $\gamma e_1=be_1-ae_2$ and $\gamma e_2=ce_1$.  Let $c_S(\gamma;T)=T^2-tT+n \in R[T]$.  By the trace hypothesis,
\[
t=\Tr_S(\gamma)=\Tr_I(\gamma)=b.
\]
Applying the relation $\gamma^2-t\gamma+n=0$ to $e_2$, and substituting above, we get
\begin{equation}
\gamma^2e_2=\gamma(ce_1)=c(be_1-ae_2)=bce_1-ace_2
\end{equation}
therefore
\begin{equation}
0=(bce_1-ace_2)-tce_1+ne_2
  =(b-t)ce_1+(n-ac)e_2.
\end{equation}
Since $e_1,e_2$ is a basis (or using that $t=b$), this gives $n=ac$.  This shows at least that $c_S(\gamma;T)=c_I(\gamma;T)=T^2-bT+ac$.  

Now let $\alpha=x+y\gamma\in S$.  On $I$, its matrix is
\[
[\alpha]_I
=
x+y[\gamma]_I
=
\begin{pmatrix}
x+by & cy\\
-ay & x
\end{pmatrix}.
\]
Hence $\Tr_I(\alpha)=2x+by$ and $\Nm_I(\alpha)=x^2+bxy+acy^2$.  On the other hand, since $\gamma^2-b\gamma+ac=0$, the norm of $\alpha=x+y\gamma$ in $S$ is
\[ \Nm_S(\alpha)
=
(x+y\gamma)(x+y(b-\gamma))
=
x^2+bxy+acy^2.
\]
Thus $\Nm_S(\alpha)=\Nm_I(\alpha)$ for all $\alpha\in S$.  Thus $c_S(\alpha;T)=c_I(\alpha;T)$ for all $\alpha\in S$, so $I$ is pseudoregular.
\end{proof}

Let $\catGoodBPSReg$ be the category of good-framed pseudoregular modules over quadratic algebras.  

\begin{cor}\label{cor: universal quadratic algebra}
    Let $R \colonequals \Z[t,n,a,b,c]/\langle t-b, n-ac \rangle \cong \Z[a,b,c]$. Then $\catGoodBPSReg$ is represented by $\Spec R \simeq \mathbb{A}^3$, with universal object $(S,I)$ where 
    \[ S=R[\alpha] \simeq R[x]/\langle x^2-tx+n\rangle = R[x]/\langle x^2-bx+ac \rangle \] 
    and $I=Re_1 \oplus Re_2$ with $[\alpha]_{(e_1,e_2)}=\begin{pmatrix} b & c \\ -a & 0 \end{pmatrix}$.  
\end{cor}

\begin{proof}
Proven in \Cref{prop:traceable} (as in \cite{Wood}*{proof of Theorem 1.4}): there are no automorphisms of a framed pseudoregular quadratic module and each object is uniquely determined by the data $t,n,a,b,c$; the relations come from the fact that the characteristic polynomial matches, so $b=t$ and $n=ac$.  
\end{proof}

We recall  $\catBQMod \simeq [\mathbb{A}^3/(\GL_2 \times \Gm)]$ \eqref{catbqmodeq} (action given in \eqref{eqn:AUscal}) and the Clifford functor (\Cref{Clif functor similarity}).

\begin{prop} \label{catquotient}
The Clifford functor induces an equivalence of quotient stacks 
\[ \catBQMod \xrightarrow{\sim} \catPsReg_2\] 
over $\catQuad$ that preserves discriminants of objects.
\end{prop}

\begin{proof}
In light of \Cref{lem:goodclifflem} and \Cref{cor: universal quadratic algebra}, it suffices to show that the action of $\GL_2 \times \Gm$ on $\catGoodBPSReg$ matches that of \eqref{eqn:AUscal}.  Indeed, $\GL_2$ still acts by the symmetric square representation (the odd Clifford bimodule is just the original binary quadratic module).  And $\Gm$ acts by the same (rescaling) character $(a,b,c) \cdot u = (ua,ub,uc)$ for $\alpha \leftarrow u\alpha$ as in \eqref{eqn:uraction}.  

The discriminant of the universal framed binary quadratic module is $b^2-4ac$ (\Cref{exm:b24ac}) matching the discriminant $t^2-4n=t^2-4ac$ of the universal framed quadratic ring (\Cref{exm:t24n}) as part of a good frame, and these descend to the quotient.
\end{proof}

\Cref{catquotient} proves \Cref{Main Thm GO}, but without the norm as an inverse---we turn to this in the next section. 

\section{Norms}  \label{sec:norms}

In this section, we define the norm functor as a pseudoinverse to the Clifford functor, establishing its properties and comparing it to the functor of Wood.

\subsection{Norm functor}

Let $\calO_Y$ be a quadratic $\calO_X$-algebra and let $\scrI$ be a $\calO_Y$-module which is locally free as an $\calO_X$-module of rank $2$ (not yet pseudoregular).  

\begin{prop} \label{prop:norm}
There exists a unique quadratic $\scrO_X$-module 
\[ E=E_{\scrI} \colon \scrI \otimes \scrO_Y \to \tbigwedge^2 \scrI \otimes \scrO_Y/\scrO_X \]
with the property that
\begin{equation} \label{eqn:conduniq}
E(x \otimes \gamma) = (\gamma x \wedge x) \otimes \overline{\gamma}
\end{equation}
for all $x \in \scrI(U)$ and $\gamma \in \scrO_Y(U)$ with $U \subseteq X$ open.  The quadratic module $E$ descends to a binary quadratic module
\[ E \colon \scrI \otimes \scrO_Y/\scrO_X \to \tbigwedge^2 \scrI \otimes \scrO_Y/\scrO_X. \]
\end{prop}

\begin{proof}
It is enough to show this locally, gluing by uniqueness.  So we work over $X=\Spec R$ with $S$ a free quadratic $R$-algebra, and $I$ an $S$-module free of rank $2$ over $R$.  Let $I=Re_1 \oplus Re_2$ and $S=R \oplus R\alpha$, and let $x \in I$ and $\gamma=r+s\alpha \in S$ with $r,s \in R$.  Since 
\begin{equation} 
\gamma x \wedge x = (r+s\alpha)x \wedge x = (s\alpha)x \wedge x = s(\alpha x \wedge x) 
\end{equation}
the map indeed descends to $I \otimes S/R$.
\begin{equation} 
I \otimes S/R \simeq R(e_1 \otimes \alpha) \oplus R(e_2 \otimes \alpha)
\end{equation}
and an arbitrary element is of the form $x \otimes \alpha$ with $x=x_1 e_1 + x_2 e_2$.  In particular, the map is uniquely defined by \eqref{eqn:conduniq}.  We similarly have 
\begin{equation} 
\tbigwedge^2 I \otimes S/R = R((e_1 \wedge e_2) \otimes \alpha)
\end{equation}
and
\[ E(x \otimes \alpha) = (\alpha x \wedge x) \otimes \alpha \]
as a map of sets $E \colon I \otimes S/R \to \tbigwedge^2 I \otimes S/R$.

To see it as an explicit quadratic map, we may suppose as in \Cref{lem:goodclifflem} that $\alpha,e_1,e_2$ is a good frame, so $[\alpha]_\beta = \begin{pmatrix} b & c \\ -a & 0 \end{pmatrix}$ with $a,b,c \in R$.  
Identifying standard bases as usual, we have 
\begin{equation} \label{eqn:x2yw}
\begin{aligned}
\alpha x \wedge x &=[\alpha]_\beta  [x]_\beta \wedge [x]_\beta = \begin{pmatrix} bx_1+cx_2 \\ -ax_1 \end{pmatrix} \wedge \begin{pmatrix} x_1 \\ x_2 \end{pmatrix}\\
& = (bx_1+cx_2)x_2 - (-ax_1)(x_1))(e_1 \wedge e_2) = E(x_1,x_2)(e_1 \wedge e_2) 
\end{aligned}
\end{equation}
where
\begin{equation} \label{eqn:e1e2}
E(x_1,x_2) \colonequals ax_1^2+bx_1x_2+cx_2^2
\end{equation}
is indeed a quadratic form, giving
\begin{equation} \label{eqn:x2yzq0}
E(x \otimes \alpha) = E(x_1,x_2) ((e_1 \wedge e_2) \otimes \alpha)
\end{equation}
and finishing the proof.
\end{proof}

We call the map $E$ in \Cref{prop:norm} the \defi{canonical exterior form} (associated to $\scrI$, over $\scrO_Y$).  To restore the domain as $\scrI$, we twist; let 
\begin{equation}  \label{eqn:Nipic}
\scrN(\scrI) \colonequals \tbigwedge^2 \scrI \otimes (\scrO_Y/\scrO_X)^{\vee} \cong (\tbigwedge^2 \scrI \otimes \scrO_Y/\scrO_X) \otimes (\scrO_Y/\scrO_X)^{\vee 2}.
\end{equation}
We define the \defi{norm map} to be the twist by $(\scrO_Y/\scrO_X)^\vee$:
\begin{equation}
N_\scrI \colonequals E_\scrI \otimes (\scrO_Y/\scrO_X)^\vee \colon \scrI \to \scrN(\scrI)
\end{equation}

It is straightforward to check that $\scrN$ defines a functor $\scrN \colon \catPsReg_2 \to \catPic$.  

\begin{remark}
The norm functor is defined on all $\scrO_X$-rank balanced $\scrO_Y$-modules, not necessarily pseudoregular.  This is analogous to the role of the \emph{exceptional rings} (which are not quaternion rings) in the case of ternary quadratic forms \cites{GL,Voight:quatring}; this suggests a generalization in higher rank.  
\end{remark}

The following lemma shows the value modules are compatible with the tensor product, at least in the invertible case.

\begin{lem} \label{lem:toomanywedges}
Suppose $\scrI,\scrI'$ are invertible over $\scrO_Y$.  Then there is a unique $\scrO_X$-module isomorphism 
\begin{equation} \label{eqn:omega}
\tbigwedge_{\scrO_X}^2 (\scrI) \otimes_{\scrO_X} \tbigwedge_{\scrO_X}^2 (\scrI') \xrightarrow{\sim} \tbigwedge_{\scrO_X}^2 (\scrI \otimes_{\scrO_Y} \scrI') \otimes_{\scrO_X} \tbigwedge_{\scrO_X}^2 (\scrO_Y) 
\end{equation}
with the property that
\begin{equation} \label{eqn:xwedgexxp}
(\gamma x \wedge x) \otimes (\gamma x' \wedge x') \mapsto \bigl(\gamma (x \otimes x') \wedge (x \otimes x')\bigr) \otimes (1 \wedge \gamma)
\end{equation}
for all $x \in \scrI(U)$, $x' \in \scrI'(U)$, and $\gamma \in \scrO_Y(U)$.  In particular, the map \eqref{eqn:omega} induces a canonical isomorphism
\begin{equation}  \label{eqn:secondstatni}
\scrN(\scrI) \otimes_{\scrO_X} \scrN(\scrI') \cong \scrN(\scrI \otimes_{\scrO_Y} \scrI').
\end{equation}
In particular, $\scrN \colon \Pic Y \to \Pic X$ is a group homomorphism.  
\end{lem}

\begin{proof}
This statement could be seen as part of a much more general formalism of norm and determinant functors \cite{GNR}; to be self-contained we sketch a direct proof.  We first prove the statement in the local case; then by uniqueness, the map glues.  So let $S=R \oplus R \alpha$ be a free quadratic $R$-algebra and let $I=Se$ and $I'=Se'$ be free $S$-modules of rank $1$.  Then  
$\tbigwedge^2 I \otimes \tbigwedge^2 I' \simeq R(\omega_I \otimes \omega_{I'})$ where $\omega_I \colonequals \alpha e \wedge e$ and similarly $\omega_{I'} \colonequals \alpha e' \wedge e'$.  We then also have $\tbigwedge^2 (I \otimes_S I') \otimes \tbigwedge^2 S \simeq R(\omega_{I \otimes I'} \otimes \omega_S)$ where $\omega_S \colonequals 1 \wedge \alpha$.  The proposed map \eqref{eqn:xwedgexxp} is defined by $\omega_I \otimes \omega_{I'} \mapsto \omega_{I \otimes I'} \otimes \omega_S$, and we observe immediately that it is an $R$-module isomorphism.  To conclude then we show it is well-defined.  First, replacing $\alpha \leftarrow u\alpha+r$ with $u \in R^\times$ and $r \in R$, all of $\omega_{I},\omega_{I'},\omega_{I \otimes I'},\omega_S$ are scaled by $u^2$, giving the same map.  If we replace $e \leftarrow \gamma e$ with $\gamma \in S^\times$, since e.g.\ $\alpha(\gamma e \wedge \gamma e) = (\det \gamma)(\alpha e \wedge e)$ we rescale both $\omega_{I}$ and $\omega_{I \otimes I'}$ by $\det \gamma$, again giving the same map; and the same with replacing $e' \leftarrow \gamma' e'$.  

The second statement \eqref{eqn:secondstatni} follows from the canonical isomorphism $\tbigwedge_{\scrO_X}^2 \scrO_Y \cong \scrO_Y/\scrO_X$ after tensoring with $((\scrO_Y/\scrO_X)^\vee)^{\otimes 2}$.  
\end{proof}

\begin{lem} \label{lem:isnorm}
If $\scrI$ is pseudoregular, then the norm map has 
\begin{equation} \label{eqn:NIAx}
N_{\scrI}(\gamma x) = \Nm(\gamma) N_{\scrI}(x)
\end{equation} 
for all $x \in \scrI(U)$ and $\gamma \in \calO_Y(U)$.
\end{lem}

\begin{proof}
The twist in $N$ does not affect property \eqref{eqn:NIAx}, so it is equivalent to show it for $E_{\scrI}$.  We can check this property locally, so we refer to \eqref{eqn:x2yzq0}. Suppose that $\alpha, e_1, e_2$ is a good frame for $S$ and $I$ with $\beta=\{1,\alpha\}$. Let $\gamma=s+r\alpha$ with $r,s \in R$, so that $[\gamma]_{\beta} = s+r[\alpha]_{\beta}$ and 
\begin{equation} 
[\gamma]_{\beta} \begin{pmatrix} x_1 \\ x_2 \end{pmatrix} = 
\begin{pmatrix} (rb+s)x_1+rcx_2 \\ -rax_1+sx_2 \end{pmatrix} = \begin{pmatrix} x_1' \\ x_2' \end{pmatrix}
\end{equation}  
Then we have
\[ E(\gamma x \otimes \alpha) = E(x_1',x_2') ((e_1 \wedge e_2) \otimes \alpha), \]
and thus
\begin{equation}
\begin{aligned}
E(x_1',x_2') &= a(x_1')^2 + bx_1'x_2' + c(x_2')^2 \\
  &= 
(acr^2+brs + s^2)(ax_1^2+bx_1x_2+cx_2^2) \\
&= \Nm(r\alpha+s)E(x_1,x_2)
\end{aligned}
\end{equation}
where the final equality uses pseudoregularity.  
\end{proof}

\begin{exm} \label{exm:ournormisanorm}
Suppose $\scrI=\scrO_Y$ is trivial.  Then $\scrN(\scrI)=\tbigwedge^2 \scrI \otimes (\scrO_Y/\scrO_X)^\vee = \tbigwedge^2 \scrO_Y \otimes (\scrO_Y/\scrO_X)^\vee \cong \scrO_X$.  
\Cref{lem:isnorm} then implies that $N_\scrI$ is just the usual norm $\Nm \colon \scrO_Y \to \scrO_X$.  
\end{exm}

\begin{cor} \label{cor:norm}
If $Q \colon \scrM \to \scrO_X$ represents $u \in \scrO_X(X)^\times$, then $Q$ is similar to the norm on $\Clf^0(Q)$.
\end{cor}

\begin{proof}
After rescaling by $u$, we may suppose $Q(e)=1$ with $e \in \scrM(X)$.  We then check that $\scrM=\scrO_Y e$ and apply \Cref{lem:isnorm}.
\end{proof}

\begin{remark}
In \cite{Kneser}, a quadratic $R$-module $Q\colon M\to L$ is called \defi{of type $C$} if $C$ is a quadratic $R$-algebra, $M$ is a projective $C$-module of rank 1, and $Q(c\alpha)=\nrd(c)Q(\alpha)$ for all $c\in C$ and $\alpha\in M$, where $\nrd$ is the reduced norm on $C$.
Kneser shows the following.  Suppose that $\Ann(Q(M))\colonequals\{r\in R\colon rQ(M)=0\}=\{0\}$. Then there is a unique $R$-algebra homomorphism $\beta\colon\Clf^0(Q)\to C$ such that $\beta(c)\alpha=c\alpha$ for all $c\in\Clf^0(Q)$ and $\alpha\in M$. If $Q$ is primitive, then $\beta$ is an isomorphism. Our notion basically equips the modules with the type given by the Clifford map.
\end{remark}

\subsection{Main theorem} \label{sec:mainthem}

\begin{thm} \label{normfunct}
The norm $\scrI \mapsto N_\scrI$ defines a functor $\catPSReg_2 \to \catBQMod$ over $\catQuad$, namely from
\begin{center}
pseudoregular modules over quadratic algebras under semilinear module isomorphisms
\end{center}
to
\begin{center}
binary quadratic modules under similarities. 
\end{center}
\end{thm}

In fact, as we saw above the norm is defined on the larger category of rank-balanced modules over quadratic algebras.

\begin{proof}
In \Cref{prop:norm}, we see that objects correspond to objects. To check morphisms, let $\psi \colon \scrI \to \scrI'$ be a semilinear $\scrO_Y$-module isomorphism with respect to $\sigma\in \Aut_{X}(Y)$; we need to check that the diagram
\begin{equation}
\begin{tikzcd}
	\scrI & \scrN(\scrI) \\
	\scrI' & \scrN(\scrI')
	\arrow["{N_\scrI}", from=1-1, to=1-2]
	\arrow["{\psi}", from=1-1, to=2-1]
	\arrow["{N_{\scrI'}}", from=2-1, to=2-2]
	\arrow["{\wedge^2\psi \otimes \sigma^\vee}", from=1-2, to=2-2]
\end{tikzcd}
\end{equation}
commutes so that $(\psi,\wedge^2 \psi \otimes \sigma^\vee)$ defines a similarity.  It is enough to check this for $E$ instead, to do so locally, and confirm
\begin{equation}
\begin{aligned}
(\wedge^2 \psi \otimes \sigma)(E(x \otimes \alpha)) &= (\psi(\alpha x) \wedge \psi(x)) \otimes \sigma(\alpha) \\ &= (\sigma(\alpha) \psi(x) \wedge \psi(x)) \otimes \sigma(\alpha) = E'(\psi(x) \otimes \sigma(\alpha))
\end{aligned}
\end{equation}
as in the proof of \Cref{prop:norm}.  

\end{proof}

We now conclude the proof of the first main theorem, which we restate.

\begin{thm} \label{Main Thm GO-inpaper}
The Clifford and norm functors define an equivalence of stacks  
\begin{equation}  
\catBQMod \xrightarrow{\sim} \catBPSReg 
\end{equation}
over $\catQuad$.  This equivalence preserves discriminants of objects.   
\end{thm}

\begin{proof}
The statement for just the Clifford functor was proven in \Cref{catquotient} and for the norm in \Cref{normfunct}.  To show that they are mutually inverse, we simply observe that on the universal good-framed object in \Cref{cor: universal quadratic algebra}, we computed in \eqref{eqn:e1e2} that the resulting norm is exactly $Q=[a,b,c]$.  
\end{proof}

\begin{cor}\label{cor: framed equivalence GSO}
    The Clifford and norm functors restrict to an equivalence of stacks $$\Qfr_2 \xrightarrow{\sim} \catPsReg_2^{\textup{gfr}}, $$
    both functors being represented by $\mathbb{A}^3=\Spec\Z[a,b,c]$.
\end{cor}
\begin{proof}
    Immediate from \Cref{Main Thm GSO} and \Cref{cor: universal quadratic algebra}.
\end{proof}

\subsection{Involutions}

In this section, we discuss two natural involutions on the category of pseudoregular $\scrO_Y$-modules and how they appear in framed coordinates.

We begin with the dual.  Let $\scrI$ be a rank-balanced $\scrO_Y$-module.  At the risk of abusing notation, we define $\scrI^\vee\colonequals \HHom_{\scrO_X}(\scrI,\scrO_X)$, also a rank-balanced $\scrO_Y$-module.  When $\scrI$ is invertible over $\scrO_Y$, then $\scrI^\vee \cong \HHom_{\scrO_Y}(\scrI,\scrO_Y)$.  

\begin{lem}\label{dual is pseudoregular}
Suppose $\scrI$ is pseudoregular.  Then $\scrI^\vee$ is pseudoregular.  
\end{lem}

\begin{proof}
Since pseudoregularity is a local property, we may work locally on an affine open subset. Let $R$ be a commutative ring, $S=R\oplus R\alpha$ be a quadratic $R$-algebra, and $I=Re_1\oplus Re_2$ be a pseudoregular $S$-module. Then $I^\vee=\Hom_R(I, R)=Re^\vee_1\oplus Re^\vee_2$ where $e^\vee_j(e_i)=\delta_{ij}.$ Let $\lambda_I\colon S \to \End_R(I)\simeq M_2(R)$ and $\lambda_{I^\vee}\colon S\to \End_R(I^\vee)\simeq M_2(R)$. For each $x\in S$, we have $\lambda_I(x)=\lambda^\intercal_{I^\vee}(x)$ and thus $c_S(x; T)=c_I(x; T)=c_{I^\vee}(x;T). $ This implies that $I^\vee$ is pseudoregular. 
\end{proof}

\begin{prop}\label{dual functor autoequivalence}
The map $\scrI \mapsto \scrI^\vee$ is a contravariant autoequivalence of $\catPSReg$.  
\end{prop}

\begin{proof}
By \Cref{dual is pseudoregular}, objects are mapped to objects. A given morphism $f\colon \scrI\to \scrJ$ of $\scrO_Y$-modules is mapped to the dual morphism $f^\vee \colon \scrJ^\vee\to \scrI^\vee$ given by $f^\vee(\phi)=\phi\circ f$. There is a canonical evaluation $\scrO_X$-module isomorphism $\ev\colon \scrI \to \scrI^{\vee\vee}$ given by $\ev(x)=(\phi\mapsto \phi(x))$. Hence, the map is a contravariant equivalence. 
\end{proof}

\begin{remark}
    The dual functor is contravariant, so it reverses the direction of the morphisms.  However, every morphism (e.g. $\scrO_Y$-module isomorphism or similarity) we use is an isomorphism, and thus this yields compatible morphisms under duality. 
\end{remark}

Next, let $\bar{}\hspace{0.1cm}\colon \scrO_Y\to \scrO_Y$ be the standard involution of $\scrO_Y$. We define the conjugate twist $\olsi{\scrI}$ to be the same underlying $\scrO_X$-module, but with $\scrO_Y$-action $s\cdot x=\bar{s}x$ for all $s\in \scrO_Y$ and $x\in \scrI$.

\begin{lem}\label{conjugate is pseudoregular}
    Suppose that $\scrI$ is pseudoregular. Then $\olsi{\scrI}$ is pseudoregular. 
\end{lem}
\begin{proof}
We work locally again. Let $R$ be a commutative ring, $S=R\oplus R\alpha$ be a quadratic $R$-algebra with standard involution $\sigma$, and $I=Re_1\oplus Re_2$ be a pseudoregular $S$-module. For each $x\in S$, we have $$c_{I^\sigma}(x;T)=c_I(\bar{x};T)=c_S(\bar{x};T)=c_S(x;T)$$ since $x$ and $\bar{x}$ have the same trace and norm in $S$. Hence, $I$ is pseudoregular. 
\end{proof}

\begin{prop}\label{bar functor autoequivalence}
The map $\scrI \mapsto \olsi{\scrI}$ is a covariant autoequivalence of the category of pseudoregular $\scrO_Y$-modules.  
\end{prop}
\begin{proof}
    By \Cref{conjugate is pseudoregular}, objects are mapped to objects. For a given morphism $f\colon \scrI\to \scrJ$ of $\scrO_Y$-modules, since $$f(s\cdot x)=f(\bar{s}x)=\bar{s}f(x)=s\cdot f(x),$$ the same underlying $\scrO_X$-linear map $f\colon \olsi{\scrI}\to \olsi{\scrJ}$ is also $\scrO_Y$-linear, and thus the construction is functorial. Also, we have the canonical identification $\olsi{\olsi{\scrI}}=\scrI$. 
\end{proof}

The two involutions above commute, i.e., 
$${\olsi{\scrI}}^\vee\simeq\olsi{\scrI^\vee}.$$ Indeed, both sides have the same underlying $\scrO_X$-module $\HHom_{\scrO_X}(\scrI, \scrO_X)$, and the $\scrO_Y$-action is given by postcomposition with $\sigma$. Thus, up to the variance of the duality functor, we obtain four operations: 
$$\scrI, \olsi{\scrI}, \scrI^\vee, \olsi{\scrI}^\vee.$$

We now spell out how involutions operate on the associated binary quadratic modules in the framed situation. Let $S=R\oplus R\alpha$ be a free quadratic $R$-algebra with $\alpha^2-b\alpha+ac=0$. Let $I=Re_1\oplus Re_2$ be a pseudoregular $S$-module with good frame $\alpha, e_1, e_2$ such that $[\alpha]_{(e_1, e_2)}=\begin{pmatrix}
    b & c\\
    -a & 0
\end{pmatrix}$. Then the associated canonical exterior form is $E_I=[a,b,c]$. In $\olsi{I}$, the same algebra generator $\alpha$ acts as $b-\alpha$ acted on $I$. To restore a good frame, we set $f_1=e_2$ and $f_2=e_1$. This gives $[\alpha]_{(f_1, f_2)}=\begin{pmatrix}
    b & a\\
    -c & 0
\end{pmatrix}$, which yields $E_{\olsi{I}}=[c,b,a]$.

A good frame for $S$ and $I^\vee$ is $\alpha, e^\vee_1, e^\vee_2$ with $[\alpha]_{(e^\vee_1, e^\vee_2)}=\begin{pmatrix}
    b & -a\\
    c & 0\\
\end{pmatrix}$, which gives $E_{I^\vee}=[-c, b, -a]$. Twisting by conjugation yields $E_{\olsi{I}^\vee}=[-a,b,-c]$.

The change of frame given by $(x,y)\mapsto (-y,x)$ (without changing the $S$-orientation) gives an oriented isometry between $E_I$ and $[c,-b,a]$, and scaling by $-1$ we obtain $E_{I^\vee}$.  In other words, $E_I$ is oriented similar to $E_{I^\vee}$.

We could also do this for $E_{\olsi{I}}=[c,b,a] \simeq [a,-b,c]$ and $E_{\olsi{I}^\vee}$.  However, when the similitude factor is restricted as in  \Cref{exm:narrowclass}, this operation is no longer allowed.

\begin{remark}
Suppose now that $X=\Spec R$ with $R$ a Dedekind domain and $F$ its field of fractions. Let $S$ be a quadratic $R$-order with field of fractions $K$, and let $I$ be an invertible fractional $S$-ideal. Then the two involutions yield the conjugate ideal $\olsi{I}$ and the inverse ideal $I^{-1}$. 

Given a prime $\mathfrak{p}\subseteq R$, we have $I_\mathfrak{p}=\alpha_\mathfrak{p}S_\mathfrak{p}$ for some $\alpha_\mathfrak{p}\in K_\mathfrak{p}$. Multiplication by $\alpha_\mathfrak{p}$ gives an $R_\mathfrak{p}$-linear isomorphism $m_{\alpha_\mathfrak{p}}\colon S_\mathfrak{p}\xrightarrow{\sim}I_\mathfrak{p}$. Taking the determinant yields $$\tbigwedge^2 I_\mathfrak{p}\simeq \det(m_{\alpha_\mathfrak{p}})\tbigwedge^2 S_\mathfrak{p}=\Nm_{K/F}(\alpha_\mathfrak{p})\tbigwedge^2 S_\mathfrak{p}.$$ Therefore, we have $$\scrN(I)_\mathfrak{p}=\tbigwedge^2 I_\mathfrak{p}\otimes (S_\mathfrak{p}/R_\mathfrak{p})^\vee\simeq \Nm_{K/F}(\alpha_\mathfrak{p})(\tbigwedge^2S_\mathfrak{p}\otimes (S_\mathfrak{p}/R_\mathfrak{p})^\vee)\simeq \Nm_{K/F}(\alpha_\mathfrak{p})R_\mathfrak{p},$$ and thus $$\scrN(I)_\mathfrak{p}\otimes_{R_\mathfrak{p}} S_\mathfrak{p}\simeq \Nm_{K/F}(\alpha_\mathfrak{p})S_\mathfrak{p}=I_\mathfrak{p}\olsi{I}_\mathfrak{p}.$$ This globalizes to $$I\olsi{I}\simeq\Nm_{K/F}(I)\otimes_RS,$$ so $\olsi{I}\simeq \Nm_{K/F}(I)I^{-1}$. The conjugate ideal coincides with the inverse only after trivializing the norm ideal $\Nm_{K/F}(I)$. In particular, over $R=\Z$ where $\Nm_{K/\Q}(I)$ is always principal, conjugation induces inversion, but this does not hold over a general Dedekind domain. 
\end{remark}

\subsection{Comparison to Wood} \label{sec:wood}

Wood \cite{Wood} defines a \defi{linear binary quadratic form} as a global section $f\in\Sym^2M\otimes L$, and Wood's Proposition 6.1 shows that $f$ yields a corresponding quadratic $\scrO_X$-module $Q\colon\scrM^\vee\to L$. The dual ensures the preservation of the discriminant between $f$ and its corresponding quadratic algebra (see \cite{Wood}*{Theorem 1.6}).  We therefore anticipate, in comparing to Wood, we will need to replace $\scrM$ with $\scrM^\vee$, see \eqref{eqn:sym2sym2}.

Let $\scrL \colonequals (\bigwedge^2\scrI)^\vee\otimes(\scrO_Y/\scrO_X)^\vee$. 
Wood \cite{Wood}*{(1)} defines the map \begin{equation}
\label{Woodeq1}
\begin{aligned}
\Psi \colon \scrL^\vee &\to \Sym^2\scrI \\
(x\wedge y)\otimes\gamma &\mapsto \gamma y\otimes x-\gamma x\otimes y
\end{aligned}
\end{equation}
As in \cref{sec:quadmodules}, we translate this map into our (more classical) notion of binary quadratic modules. Taking the dual yields a map $\Psi^\vee\colon (\Sym^2\scrI)^\vee \to \scrL^{\vee\vee}$ given by $\Psi^\vee(f)(\theta)=f(\Psi(\theta))$ for all $f\in (\Sym^2\scrI)^\vee$ and $\theta \in \scrL^\vee$. The canonical isomorphism $(\Sym_2(\scrI^\vee))^\vee\cong (\Sym^2\scrI)^\vee$ yields a binary quadratic $\scrO_X$-module $Q_{\scrI^\vee}\colon \scrI^\vee\to \scrL$. The map $\scrI\mapsto Q_{\scrI^\vee}$ defines a contravariant functor which maps an $\scrO_Y$-module isomorphism $\psi\colon \scrI\to \scrI'$ to the similarity $(\psi^\vee, (\wedge^2\psi)^\vee\otimes \id)$ from $Q_{\scrI'^\vee}$ to $Q_{\scrI^\vee}$. 

\begin{remark}
\eqref{Woodeq1} allows us to directly obtain the corresponding binary quadratic module without passing through induced linear binary quadratic forms via \cite{Wood}*{Proposition 6.1}.
\end{remark}

\begin{prop} 
The composition of the dual functor and the norm functor is naturally isomorphic to the functor $\scrI \to Q_{\scrI^\vee}$, both functors being contravariant.
\end{prop}

\begin{proof} 
Let $\xi_{\scrI} \colon \tbigwedge^2 \scrI^\vee\to (\tbigwedge^2 \scrI)^\vee$ be the $\scrO_X$-module isomorphism defined by 
\begin{equation}
\xi_\scrI(\theta \wedge \eta)=u\wedge v\mapsto \theta(v)\eta(u)-\theta(u)\eta(v). 
\end{equation}
We claim that $(\id, \xi\otimes \id)$ is a similarity from $N_{\scrI^\vee}$ to $Q_\scrI$, i.e. the diagram \[\begin{tikzcd}
	\scrI^\vee & \tbigwedge^2\scrI^\vee \otimes (\scrO_Y/\scrO_X)^\vee\\
	\scrI^\vee & (\tbigwedge^2\scrI)^\vee\otimes(\scrO_Y/\scrO_X)^\vee
	\arrow["{N_{\scrI^\vee}}", from=1-1, to=1-2]
	\arrow["{\id}", from=1-1, to=2-1]
	\arrow["{Q}", from=2-1, to=2-2]
	\arrow["{\xi_\scrI\otimes \id}", from=1-2, to=2-2]
\end{tikzcd}\]
commutes. It is enough to show this locally. Let $U=\Spec R$. Let $S=\scrO_Y(U)$ be a free quadratic $R$-algebra, and let $I=\scrI(U)$ be an $S$-module free of rank 2 over $R$ with good frame $\alpha, e_1, e_2$ such that $[\alpha]_{(e_1,e_2)}=\begin{pmatrix} b & c \\ -a & 0 \end{pmatrix}$. Then $\alpha, e^\vee_1, e^\vee_2$ is a good frame for $S$ and $I^\vee$, and we have $[\alpha]_{(e^\vee_1, e^\vee_2)} = \begin{pmatrix} b & -a \\ c & 0 \end{pmatrix}$, which gives  $N_{I^\vee}(x,y)=(-cx^2+bxy-ay^2)(e^\vee_1\wedge e^\vee_2)\otimes\alpha^\vee$.

Let $L=(\tbigwedge^2I)^\vee\otimes (S/R)^\vee$. Applying \eqref{Woodeq1}, we get the map $\Psi|_U \colon L^\vee\to \Sym^2 I$ given by  $\Psi|_U((e_1\wedge e_2)\otimes \alpha)=\alpha e_2\otimes e_1-\alpha e_1\otimes e_2=ce_1\otimes e_1-be_1\otimes e_2+ae_2\otimes e_2$. This yields the map $\Psi^\vee|_U\colon (\Sym^2I)^\vee\to L$ given by $\Psi^\vee|_U((e_1\otimes e_1)^\vee)=c(e_1 \wedge e_2)^\vee\otimes \alpha^\vee$, $\Psi^\vee|_U((e_1\otimes e_2)^\vee)=-b(e_1\wedge e_2)^\vee\otimes \alpha^\vee$, and $\Psi^\vee|_U((e_2\otimes e_2)^\vee)=a(e_1 \wedge e_2)^\vee\otimes \alpha^\vee$. From the isomorphism $(\Sym_2(I^\vee))^\vee\cong (\Sym^2 I)^\vee$, we obtain the quadratic module $Q\colon I^\vee\to L$ given by $Q(xe^\vee_1+ye^\vee_2)=(cx^2-bxy+ay^2)(e_1\wedge e_2)^\vee\otimes \alpha^\vee$. Since $\xi_I|_U(e^\vee_1\wedge e^\vee_2)=-(e_1\wedge e_2)^\vee$, we have $(\xi_I|_U\otimes \id)(N_{I^\vee}(x,y))=Q(x,y)$.

Next, let $\scrI'$ be another $\scrO_Y$-module which is locally free of rank 2 as an $\scrO_X$-module. Let $\psi\colon \scrI\to \scrI'$ be an $\scrO_Y$-module isomorphism. By \Cref{normfunct}, the pair $(\psi^\vee, \wedge^2\psi^\vee\otimes \id)$ defines a similarity from $N_{\scrI'^\vee}$ to $N_{\scrI^\vee}$. We want to show that the diagram 
\[\begin{tikzcd}
	N_{\scrI'^\vee} &&& Q_{\scrI'} \\
	\\
	N_{\scrI^\vee} &&& Q_\scrI
	\arrow["{(\id, \,  \xi_{\scrI'}\otimes\id)}", from=1-1, to=1-4]
	\arrow["{(\psi^\vee, \, \wedge^2\psi^\vee\otimes\id)}"', from=1-1, to=3-1]
	\arrow["{(\psi^\vee, \, (\wedge^2\psi)^\vee\otimes \id)}", from=1-4, to=3-4]
	\arrow["{(\id, \, \xi_\scrI\otimes \id)}"', from=3-1, to=3-4]
\end{tikzcd}\]
commutes. We have $(\psi^\vee\circ\id)(\scrI'^\vee)=\scrI^\vee=(\id\circ  \psi^\vee)(\scrI'^\vee)$ and
\begin{equation}
    \begin{aligned}
        (((\wedge^2\psi)^\vee\otimes\id)\circ (\xi_{\scrI'}\otimes\id))(\scrN(\scrI'^\vee))&=((\wedge^2\psi)^\vee\otimes \id)((\tbigwedge^2\scrI')^\vee \otimes(\scrO_Y/\scrO_X)^\vee)\\
    &= (\tbigwedge^2\scrI)^\vee\otimes (\scrO_Y/\scrO_X)^\vee\\
    &= (\xi_\scrI\otimes \id)(\tbigwedge^2\scrI^\vee\otimes(\scrO_Y/\scrO_X)^\vee)\\
    &= ((\wedge^2\psi^\vee\otimes \id)\circ(\xi_\scrI\otimes\id))(\scrN(\scrI'^\vee))
    \end{aligned}
\end{equation}
as claimed.
\end{proof}

\begin{remark}
    Although the functor defined by Wood is covariant, by \Cref{dual functor autoequivalence} we work equally well with a contravariant functor.  
\end{remark}

\section{Orientations} \label{sec:dor}

In this section, we prove the theorem corresponding to oriented similarities by studying the $2$-fiber.

\subsection{Orientations} \label{sec:orient}

We now organize quadratic modules according to the invariant provided by the universal Clifford center.  

Recall the universal framed quadratic module and its even Clifford algebra in the proof of \Cref{prop:cliffpseudo}.  Given a quadratic $\scrO_X$-module $Q\colon \scrM \to \scrL$, the center of the universal even Clifford algebra specializes to the even Clifford algebra of $Q$, which we call the \defi{universal Clifford center} of $Q$, denoted $\scrZ(Q)$.  By construction, the formation of the universal Clifford center is functorial for similarities and commutes with base change.  

\begin{lem}
If $\scrM$ is a locally free $\scrO_X$-module of even rank, then $\scrZ(Q)$ is a quadratic $\scrO_X$-algebra.  
\end{lem}

\begin{proof}
Indeed, it is enough to check this for the center of the universal even Clifford algebra: for details, see \cite{Auel15}*{Proposition 1.2}.  
\end{proof}

Therefore the even Clifford functor (\Cref{prop:evencliff}) also factors naturally over $\catQuad$.  

\begin{definition}
Let $\scrO_Y$ be a quadratic $\scrO_X$-algebra.  An $\scrO_Y$-\defi{orientation} of a quadratic $\scrO_X$-module $Q \colon \scrM \to \scrL$ is an $\calO_X$-algebra isomorphism $\iota \colon\scrZ(Q) \xrightarrow{\sim} \scrO_Y$.  An \defi{orientation} of $Q$ is an $\calO_Y$-orientation for some $Y$.  
\end{definition}

\begin{definition}
Let $Q'\colon\scrM' \to \scrL'$ be another quadratic $\scrO_X$-module with $\scrO_Y$-orientation $\iota'\colon\scrZ(Q')\xrightarrow{\sim}\scrO_Y$. We say that a similarity $(\varphi, \lambda)$ between $Q$ and $Q'$ is \defi{$\scrO_Y$-oriented} if the diagram 
\begin{equation}
\begin{tikzcd}
	\scrZ(Q) && \scrZ(Q') \\
	& \scrO_Y
	\arrow["{\Clf^0(\varphi, \lambda)}", from=1-1, to=1-3]
	\arrow["\sim", "\iota"', from=1-1, to=2-2]
	\arrow["\iota'", "\sim"', from=1-3, to=2-2]
\end{tikzcd}
\end{equation}
commutes. 
\end{definition}

\begin{example} \label{exm:canonical}
Every $Q$ has a canonical $\scrZ(Q)$-orientation given by the identity.  
\end{example}

\begin{lem}
    The set of $\scrO_Y$-orientations, if nonempty, is a torsor under $\Aut_{X}(Y)$ by postcomposition.
\end{lem}

\begin{proof}
 Given an $\scrO_Y$-orientation $\iota\colon \scrZ(Q)\to \scrO_Y$, if $u\circ \iota=\iota$ for $u\in \Aut_{X}(Y)$, then $u=\id$, so the action is free.  Given another $\scrO_Y$-orientation $\iota'\colon \scrZ(Q)\to \scrO_Y$, then $u\circ \iota =\iota'$ with $u=\iota'\circ \iota^{-1} \in \Aut_{X}(Y)$, and thus the action is transitive.
\end{proof}

\begin{exm}
If $Y$ is an integral scheme (hence $X$ is also integral) and $Y \to X$ is generically separable, then $\Aut_{X}(Y) \simeq \Z/2\Z$ with the nontrivial element given by the standard involution. 
\end{exm}

\begin{exm}
Consider $X=\Spec F$ where $F$ is a field with $\opchar F \neq 2$ and let $(V,Q)$ be a quadratic space over $F$.  If $\Clf^0(Q)\simeq F(\sqrt{d})$, then an orientation of $Q$ is a choice between $\sqrt{d}$ and $-\sqrt{d}$, as in Galois theory.
\end{exm}



We now restrict to the binary case of rank $m=2$.  Then $\scrZ(Q)=\Clf^0(Q)$, as the even Clifford algebra is already commutative.  

\begin{proof}[Proof of \Cref{Main Thm GSO}]
We take $2$-fibers in \Cref{Main Thm GO}, applying \Cref{lem:fibres-of-equivalence}.  The $2$-fiber of $\catBQMod\to\catQuad$ over $\scrO_Y$ is precisely the category of $\scrO_Y$-oriented binary quadratic $\scrO_X$-modules under oriented similarities.  Likewise, the $2$-fiber of $\catBPSReg\to\catQuad$ over $\scrO_Y$ is the category of pseudoregular $\scrO_Y$-modules under $\scrO_Y$-module isomorphisms.
\end{proof}

\begin{remark}
We could also define the total category of oriented binary quadratic modules under oriented similarities, with objects $(Q,\scrA,\iota)$ with $\iota \colon \Clf^0(Q) \xrightarrow{\sim} \scrA$ under isomorphisms, allowing the target quadratic algebra of the orientation to vary.  But this construction does not produce a new stack: the canonical orientation (\Cref{exm:canonical}) is a quasi-inverse to the forgetful functor $(Q,\scrA,\iota) \mapsto Q$.  Indeed, every oriented object is canonically isomorphic to the canonically oriented object $(Q,\Clf^0(Q),\id)$ using $\iota$.  Thus the full subcategory of canonically oriented objects is essential.  The nontrivial operation is fixing the target quadratic algebra.
\end{remark}

\subsection{Frames, and canonical orientations}

We conclude by combining frames with orientations.  
Recall the even Clifford functor defines a map of stacks $\Qfr \to \catQuad$, with the universal object $Q=[a,b,c]$ over $\Spec \Z[a,b,c]$ mapping to $\Clf^0(Q)=R[e_1e_2]$ where $e_1e_2$ satisfies $x^2-bx+ac=0$.  Recall we also have a universal framed quadratic algebra $S \colonequals \Z[t,n][x]/\langle x^2-tx+n \rangle$ over $T\colonequals \Z[t,n]$ from \eqref{eqn:quadalguniv}, which we can view as an object in $\catQuad$.  

\begin{prop} \label{prop:univorient}
The $2$-fiber of $\Qfr \to \catQuad$ over $S$ is
represented by 
\[ R=\Z[t,n][a,b,c,u,u^{-1},r]/\langle u(t+2r)-b, u^2(n+tr+r^2)-ac \rangle. \]
The universal object over $R$ is the framed binary quadratic module $Q=[a,b,c]$ with the $S_R$-orientation 
\begin{equation}
\begin{aligned}
\iota \colon \Clf^0(Q) = R[e_1e_2] &\xrightarrow{\sim} S_R = R[x]/\langle x^2-tx+n \rangle \\
e_1e_2 &\mapsto u(\alpha+r)
\end{aligned}
\end{equation}
\end{prop}

\begin{proof}
A coordinate change on the universal quadratic algebra is specified as in \eqref
{eqn:coordchange}; the equations imply the match with the even Clifford algebra.
\end{proof}

\begin{remark}
If instead we work with the fiber of $\Qfr \to \catQuad^{\textup{fr}}$, then we get the tautological representation by $\Z[t,n][a,b,c]/\langle b-t,ac-n\rangle$ as the frames must match.  
\end{remark}

We have a canonical $\calO_X$-module isomorphism 
\begin{equation} \label{eqn:canonisom}
\begin{aligned}
\Clf^0(N_\scrI)/\calO_X \cong \tbigwedge^2 \scrI \otimes \scrN(\scrI)^\vee 
= \tbigwedge^2 \scrI \otimes (\tbigwedge^2 \scrI \otimes (\scrO_Y/\scrO_X)^\vee)^\vee 
\cong \scrO_Y/\scrO_X.
\end{aligned}
\end{equation}

\begin{prop} \label{prop:canorient}
Suppose that $\scrI$ is pseudoregular.  Then there exists a unique $\calO_X$-algebra isomorphism 
\begin{equation}  \label{eqn:canorient-defn}
\Clf^0(N_\scrI) \xrightarrow{\sim} \scrO_Y 
\end{equation}
which lifts the canonical $\calO_X$-module isomorphism \eqref{eqn:canonisom}.
\end{prop}

We call the isomorphism \eqref{eqn:canorient-defn} the \defi{canonical orientation}.

\begin{proof}
We will exhibit a unique such isomorphism locally. On an open set $U=\Spec{R}$, we have $S=\scrO_Y(U)$ and $I=\scrI(U)$ with a good frame $\alpha, e_1, e_2$. Let $N_I \colon I \to L$ where $L= \tbigwedge^2 I \otimes (S/R)^\vee$.  Then $\Clf^0(N_I) \simeq R \oplus Re_1e_2$. In the canonical isomorphism \eqref{eqn:canonisom}, locally we have $\Clf^0(N_I)/R \simeq R(e_1e_2) \xrightarrow{\sim} S/R = R \alpha$ where the image of $e_1e_2$ is $u\alpha$ with $u \in R^\times$ and we rescale $e_1$ so that $u=1$.  

The map $e_1e_2 \mapsto \alpha$ in fact defines an $R$-algebra isomorphism. It is enough to show that $e_1e_2$ and $\alpha$ have the same characteristic polynomial.  From the above, using that $I$ is pseudoregular, we have $\Tr_I(\alpha)=\Tr_S(\alpha)=b$ and $\Nm_I(\alpha)=\Nm_S(\alpha)=ac$.  We similarly compute using \eqref{eqn:e1e2}
that $(e_1e_2)e_1=e_1(b-e_1e_2)=be_1-ae_2$ and $(e_1e_2)e_2 =ce_1$ giving $[e_1e_2]_{(e_1, e_2)}=\begin{pmatrix} b & c\\ -a & 0\end{pmatrix}$ which indeed matches. It is clear that the isomorphism is independent of the choice of good frame.
\end{proof}

\begin{remark}
\Cref{prop:canorient} allows us to describe explicitly the natural isomorphisms showing that the Clifford and norm functors are quasi-inverse, giving another proof of \Cref{Main Thm GO-inpaper}. Let $\mathcal{F}$ and $\mathcal{G}$ denote the Clifford and norm functors respectively. 

Let $Q \colon \scrM \to \scrL$ be a binary quadratic $\scrO_X$-module. Then $(\mathcal{G} \circ \mathcal{F})(Q)=N_{\Clf^1(Q)}$. There is an $\scrO_X$-module isomorphism $\lambda_Q$ from $\scrN(\Clf^1(Q))=\scrN(\scrM)=\bigwedge^2\scrM \otimes (\Clf^0(Q)/\scrO_X)^\vee\simeq \bigwedge^2\scrM \otimes (\bigwedge^2 \scrM \otimes \scrL^\vee)^\vee$ to $\scrL$, so $(\id, \lambda_Q)$ is a similarity  between $N_{\Clf^1(Q)}$ and $Q$.  Let the pair of $\scrO_X$-module isomorphisms $\varphi \colon \scrM \to \scrM'$ and $\lambda\colon \scrL \to \scrL'$ be a similarity between $Q$ and another binary quadratic $\scrO_X$-module $Q' \colon \scrM' \to \scrL'$. By the canonical identification in \Cref{binary even cliff mod O_X}, the diagram 
\begin{equation}
\begin{tikzcd}
 	N_{\Clf^1(Q)} &&& Q\\
     \\
 	N_{\Clf^1(Q')} &&& Q'
 	\arrow["{(\id, \lambda_Q)}", from=1-1, to=1-4]
 	\arrow["{(\mathcal{G}\circ\mathcal{F})(\varphi, \lambda)}"', from=1-1, to=3-1]
 	\arrow["{(\varphi, \lambda)}", from=1-4, to=3-4]
 	\arrow["{(\id, \lambda_{Q'})}", from=3-1, to=3-4]
 \end{tikzcd}
 \end{equation}
 commutes, so $\mathcal{G} \circ \mathcal{F}\simeq\id_{\catBQMod}$.

Conversely, for a pseudoregular $\scrO_Y$-module $\scrI$, under the identification $\iota\colon\Clf^0(N_\scrI)\xrightarrow{\sim}\scrO_Y$ in \Cref{prop:canorient}, we have $(\mathcal{F}\circ\mathcal{G})(\scrI)=\Clf^1(N_\scrI)=\scrI$. Given another pseudoregular $\scrO_{Y'}$-module $\scrI'$, let $(\sigma, \psi) \colon (\scrO_Y, \scrI) \to (\scrO_{Y'}, \scrI')$ be a semilinear isomorphism. Let $\iota'\colon\Clf^0(N_{\scrI'})\xrightarrow{\sim}\scrO_Y$ be the canonical orientation. By \eqref{eqn:canonisom}, the diagram \begin{equation}
\begin{tikzcd}
 	(\Clf^0(N_\scrI), \Clf^1(N_\scrI)) &&& (\scrO_Y, \scrI)\\
     \\
 	(\Clf^0(N_{\scrI'}), \Clf^1(N_{\scrI'})) &&& (\scrO_{Y'}, \scrI')
 	\arrow["{(\iota, \id)}", from=1-1, to=1-4]
 	\arrow["{(\mathcal{F}\circ \mathcal{G})(\sigma, \psi)=(\Clf^0(\psi, \wedge^2\psi\otimes\sigma^\vee), \psi)}"', from=1-1, to=3-1]
 	\arrow["{(\sigma, \psi)}", from=1-4, to=3-4]
 	\arrow["{(\iota', \id)}", from=3-1, to=3-4]
 \end{tikzcd}
 \end{equation}
 commutes, so $\mathcal{F}\circ \mathcal{G}\simeq\id_{\catBPSReg}$.
\end{remark}

\subsection{Primitivity and invertibility}

In this section, we restrict our equivalences to nondegenerate objects to see the resulting restricted equivalences; as applications, taking isomorphism classes of objects, we derive relationships to Picard groups.  

\begin{definition}
    We say a quadratic $\scrO_X$-module $(\scrM, \scrL, Q)$ is \defi{primitive} if for each open $U\subseteq X$, the sections $Q(x)$ with $x\in \scrM(U)$ generate $\scrL(U)$ as an $\scrO_X(U)$-module. 
\end{definition}

\begin{prop}\label{primitive = invertible}
A quadratic $\calO_X$-module is primitive if and only if it corresponds under Clifford to an invertible $\calO_Y$-module.
\end{prop}

\begin{proof}
The conditions are local on the base, so we reduce to the free case over $R$ and choose a frame.  We then refer to \Cref{exm:a3bdiso} and \Cref{exm:goodcliff}, with $Q \leftrightarrow I$.  

We may further localize so that $R$ is a local ring with maximal ideal $\mathfrak{m}$.  Suppose $Q=[a,b,c]$ is primitive; then at least one of $a,b,c \in R^\times$.  If $a \in R^\times$ then we check that $I=Se_1$; if $c \in R^\times$ then $I=Se_2$.  If $b \in R^\times$ and $a,c \in \mathfrak{m}$ then $I=S(e_1+e_2)$.  In any case, $I$ is free of rank $1$.  

Conversely, if $I=S e \simeq S$ then since the norm furnishes an inverse to Clifford, we see that $Q$ is similar to $\Nm_S$ which represents $1$, hence $Q$ is primitive.  
\end{proof}

Now we prove the corollaries with respect to the isomorphism classes. 

\begin{proof}[Proof of \Cref{GSO class =Pic Y}]
Immediate from taking isomorphism classes on both sides of \Cref{Main Thm GSO} by \Cref{primitive = invertible}.
\end{proof}

\begin{exm} \label{exm:classical}
If $X=\Spec \Z$, then $Y=\Spec S$ is a quadratic ring with discriminant $d \in \Z$.  Then \Cref{GSO class =Pic Y} says that primitive binary quadratic forms of discriminant $d$ up to $\SL_2(\Z)$-equivalence and scaling by $\pm1$ are in bijection with the class group of $S$.  The orientation corresponds to a fixed choice of generator of $S/\Z$.
\end{exm}




\section{Universal composition law}\label{sec: universal composition law}
In the previous section, we saw that the group law on the Picard group induces a composition law on classes of invertible modules.  In this section, we extend this composition law to a wide class of pairs of pseudoregular modules and express this directly on the level of quadratic modules, in particular recovering classical composition laws. 

\subsection{Coprimitivity}

We first work locally. Let $R$ be a commutative ring (with 1), let $S$ be a free quadratic $R$-algebra, and let $I$ and $I'$ be pseudoregular $S$-modules, free over $R$. We can of course form the tensor product $I \otimes_S I'$, but the result need not be pseudoregular.

\begin{exm}
Returning to \Cref{exm:apeislon}, let $S=R[\varepsilon]$ with $\varepsilon^2=0$ and $I=R^2$ having $\varepsilon$ acting by $0$.  Then $I \otimes_S I=I \otimes_R I$ is an $R$-module of rank $4$---the condition of $S$-linearity is automatic.
\end{exm}

Let $\alpha, e_1, e_2$ and $\alpha', e'_1, e'_2$ be good frames for $S, I$ and $S, I'$, respectively.  Then $\alpha'=u\alpha+r$ for some $r \in R$ and $u \in R^\times$.  We say that the pair of good frames is \defi{unital} if $u=1$; replacing $\alpha \leftarrow u\alpha$, we may suppose that we have a unital pair of good frames.  By \Cref{cor: universal quadratic algebra}, we have $\alpha^2-b\alpha+ac={\alpha'}^2-b'\alpha'+a'c'=0$ for some $a,b,c,a',b',c'\in R$.  It follows that 
\begin{equation} \label{eqn:blerg}
b'=b+2r \quad \textup{and} \quad a'c'=ac+br+r^2.
\end{equation}

\begin{definition}
We say that $I,I'$ are \defi{coprimitive} if the elements $a,b,c,a',b',c',r$ generate the unit ideal.  
\end{definition}

\begin{lem} \label{lem:coprim}
The coprimitive condition on $I$ and $I'$ is independent of the choice of unital pair of good frames.
\end{lem}

\begin{proof}
Let $\hat{\alpha}, \hat{e_1}, \hat{e_2}$ and $\hat{\alpha'}, \hat{e_1'}, \hat{e_2'}$ be another unital pair of good frames for $S, I$ and $S, I'$, giving an ideal $\langle \hat{a},\hat{b},\hat{c},\hat{a'},\hat{b'},\hat{c'},\hat{r} \rangle$.  We show that this ideal is equal to $\langle a,b,c,a',b',c',r\rangle$.  

First, by \Cref{good frame independence}, we have $\langle \hat{a},\hat{b},\hat{c}\rangle = \langle a,b,c \rangle$ and similarly with primes.  For the remaining element, we have $\hat{\alpha}=\alpha+s$ and $\hat{\alpha'}=\alpha'+s'$ for $s, s'\in R$. By \eqref{eqn:solveforr}, we have $s\in \langle a, b,c\rangle$ and $s'\in \langle a', b', c'\rangle$, so $\hat{r}=\hat{\alpha'}-\hat{\alpha}=\alpha'-\alpha+s'-s=r+s'-s\in \langle a,b,c, a', b', c', r\rangle$. 
\end{proof}

\begin{remark}
Rescaling $\alpha$ by a unit $u\in R^\times$ to obtain a unital pair does not change the ideal: $$\langle a,b,c, a', b', c', r\rangle=\langle ua, ub, uc, a', b', c', r\rangle.$$  So \Cref{lem:coprim} could be extended to any pair of good frames.  
\end{remark}

We now define the global notion. Let $X$ be a scheme, let $\scrO_Y$ be a quadratic $\scrO_X$-algebra, and let $\scrI, \scrI'$ be pseudoregular $\scrO_Y$-modules which are locally of rank 2 as $\scrO_X$-modules.
\begin{definition}
    We say $\scrI$ and $\scrI'$ are \defi{coprimitive} if there exists an affine open cover $X=\bigcup_i U_i$ such that $\scrI(U_i)$ and $\scrI'(U_i)$ are coprimitive over $\scrO_Y(U_i)$ for all $i$.
\end{definition}

It is straightforward to check that $\scrI$ and $\scrI'$ are coprimitive if and only if for all affine open $U\subseteq X$, the modules $\scrI(U)$ and $\scrI'(U)$ are coprimitive over $\scrO_Y(U)$.
\begin{lem}\label{lem: Fitting ideal}
Let $M$ be an $R$-module presented by $$R^m\xrightarrow{A}R^n\to M\to 0$$ with $A \in \operatorname{Mat}_{m,n}(R)$.  For $k\geq 0$, let $I_k(A)$ denote the ideal of $R$ generated by the $k\times k$-minors of $A$ with $I_0(A)\coloneq R$. Then $M$ is locally free of rank $r$ if and only if $I_{n-r}(A)=R$ and $I_{n-r+1}(A)=\{0\}$.
\end{lem}

\begin{proof}
    See Eisenbud \cite{Eis2013}*{Proposition 20.8}.
\end{proof}

\subsection{Tensor product and composition} \label{sec:tensorcomp}

In this section, we show that coprimitivity characterizes framed composition laws and compute the universal law on an open cover, including Dirichlet composition.

\begin{prop}\label{prop: tensor prod of coprimitive}
The tensor product $\scrI\otimes_{\scrO_Y} \scrI'$ is locally free of rank 2 as an $\scrO_X$-module if and only if $\scrI$ and $\scrI'$ are coprimitive.
\end{prop}
\begin{proof}
Being locally free as an $\scrO_X$-module is local on $X$ and may be checked on affine opens by \Cref{lem: Fitting ideal}. On the other hand, the coprimitivity condition is also local and is independent of the choice of unital good frames by \Cref{lem:coprim}. Therefore, it suffices to prove the claim after replacing $X$ by an affine open on which all relevant modules are free, and good frames are chosen. 

Let $X=\Spec{R}$, let $S=R\oplus R\alpha$ be a free quadratic $R$-algebra, and let $I, I'$ be pseudoregular $S$-modules which  are free of rank 2 as $R$-modules. Choose a unital pair of good frames $\alpha, e_1, e_2$ and $\alpha', e'_1, e'_2$ so that $\alpha'=\alpha+r$.

Now we work over the universal base ring $$R_{\univ}\colonequals\Z[a,b,c,a',b',c', r]/\langle b'-b-2r, a'c'-ac-br-r^2\rangle,$$ and every local situation arises as a unique specialization of $R_{\univ}$. We consider the tensor product $I\otimes_S I'$ with $$(s\cdot x)\otimes y=x\otimes (s\cdot y)$$ for all $s\in S, x\in I$, and $y\in I'$. The relations $$(\alpha e_i)\otimes e'_j + re_i\otimes e'_j= e_i\otimes (\alpha'e'_j)$$ for each $i, j= 1$ or $2$ give the columns of $$A\coloneqq\begin{pmatrix}
    -r & -c' & c & 0\\
    a' & b+r & 0 & c\\
    -a & 0 & -(b+r) & -c'\\
    0 & -a & a' & r
\end{pmatrix}$$
in the basis $e_1 \otimes e_1'$, $e_1 \otimes e_2'$, $e_2 \otimes e_1'$, and $e_2 \otimes e_2'$, such that $I\otimes_S I'$ is isomorphic to the quotient of $ R^4$ by the column $R$-span of $A$. Let $I_2(A)$ and $I_3(A)$ be the ideals generated by the $2\times 2$ and $3 \times 3$ minors of $A$ respectively. We compute $I_3(A)=\{0\}$ and 
\begin{equation}
\begin{aligned}
I_2(A) &=\langle a^2, aa', ab+ar, ac-a'c',ac+br+r^2,ac',ar,(a')^2,a'b+a'r,a'c,\\
&\qquad\qquad a'c'-br-r^2,a'r,b^2+2br+r^2,bc+cr,bc'+c'r,c^2,cc',cr,(c')^2,c'r,r^2 \rangle 
\end{aligned}
\end{equation}
We then observe that $\sqrt{I_2(A)}=\langle a,b,c,a',b',c',r\rangle$.  Recall that $I_2(A)=R$ if and only if $\sqrt{I_2(A)}=R$.  Thus taking $N=I\otimes_S I'$ in \Cref{lem: Fitting ideal}, we conclude that $I\otimes_S I'$ is locally free of rank 2 if and only if $\langle a,b,c,a',b',c',r\rangle=R$.
\end{proof}

\begin{thm}\label{tensor product is pseudoregular}
If $\scrI$ and $\scrI'$ are coprimitive, then $\scrI\otimes_{\scrO_Y} \scrI'$ is pseudoregular.
\end{thm}
\begin{proof}
By \Cref{prop: tensor prod of coprimitive}, $\scrI\otimes_{\scrO_Y}\scrI'$ is locally free of rank 2 as an $\scrO_X$-module. Since pseudoregularity and coprimitivity are local, it suffices to check on an affine open cover. Let $X=\Spec{R}$, let $S=R\oplus R\alpha$ be a free quadratic $R$-algebra, and let $I, I'$ be pseudoregular $S$-modules which are free of rank 2 as $R$-modules. Choose a unital pair of good frames $\alpha, e_1, e_2$ and $\alpha', e'_1, e'_2$ so that $\alpha'=\alpha+r$. Since $I$ and $I'$ are coprimitive, $\langle a, b, c, a', b', c', r\rangle=R$ by \Cref{prop: tensor prod of coprimitive}. Thus, the principal opens $$D(a), D(a'), D(c), D(c'), D(r), D(b+r)$$ 
cover $\Spec{R}$ where $D(J) = \Spec R \smallsetminus \Spec J$ for an ideal $J \subseteq R$. It suffices to verify the claim on these opens. Along the way, we also compute the norm on $I\otimes_SI'$, giving an explicit composition.
\begin{enumerate}
    \item On $D(a)$, we have $I\otimes_S I'=Se_1\otimes_S I'\simeq I'$ which is pseudoregular.  The good frame $\alpha', e_1\otimes e'_1, e_1\otimes e'_2$ yields the composition $$Q''= Q'=[ a',b',c'].$$
    \item On $D(a')$, we have $I\otimes_S I'=I\otimes_S Se'_1\simeq I$ which is pseudoregular.  The good frame $\alpha, e_1\otimes e'_1, e_2\otimes e'_1$ yields the composition $$Q''=Q=[ a,b,c].$$
    \item On $D(c)$, we have $I\otimes_S I'=Se_2\otimes_S I'\simeq I'$ which is pseudoregular.   The good frame $\alpha',  e_2\otimes e'_1, e_2\otimes e'_2$ yields the same composition as case (1).
    \item On $D(c')$, we have $I\otimes_S I'=I\otimes_S Se'_2\simeq I$ which is pseudoregular.  The good frame $\alpha, e_1\otimes e'_2, e_2\otimes e'_2$ yields the same composition as case (2). 
    \item On $D(r)$, $I\otimes_S I'$ has an $R$-basis $f_1=e_1\otimes e'_2, f_2= e_2\otimes e'_1$. We have 
    \begin{equation}
    \alpha f_1=bf_1-ae_2\otimes e'_2=bf_1+ar^{-1}(cf_1-c'f_2)=(b+acr^{-1})f_1-ac'r^{-1}f_2
    \end{equation}
    and 
    \begin{equation} 
    \alpha f_2=ce_1\otimes e'_1=ca'r^{-1}f_1-car^{-1}f_2.
    \end{equation}
    Then $\Tr_S(\alpha)=\Tr_{I\otimes_SI'}(\alpha)=b$, and thus $I\otimes_S I'$ is pseudoregular.  The good frame is $r\alpha+ac, e_1\otimes e'_2, e_2\otimes e'_1$  yields the composition $$Q''=[ ac', br+2ac, ca'].$$
    \item On $D(b+r)$, $I\otimes_S I'$ has an $R$-basis $f_1=e_1\otimes e'_1, f_2=e_2\otimes e'_2$. Indeed, we have 
    \begin{equation}
    \begin{aligned}
    \alpha f_1 &=bf_1-ae_2\otimes e'_1=bf_1-a(c(b+r)^{-1}f_1+a'(b+r)^{-1}f_2)\\
    &=(b-ac(b+r)^{-1})f_1-aa'(b+r)^{-1}f_2
    \end{aligned}
    \end{equation}
    and 
    \begin{equation}
    \alpha f_2=ce_1\otimes e'_2=cc'(b+r)^{-1}f_1+ca(b+r)^{-1}f_2.
    \end{equation}
    Then $\Tr_S(\alpha)=\Tr_{I\otimes_S I'}(\alpha)=b$, and thus $I\otimes_S I'$ is pseudoregular. The good frame is $\alpha(b+r)-ac, e_1\otimes e'_1, e_2\otimes e'_2$ yields the composition $$Q''=[aa', b(b+r)-2ac, cc'].$$ 
\end{enumerate}
These opens together form a cover, completing the proof.
\end{proof}

Amazingly, we can recover Dirichlet composition of united forms in the local case (see a brief review in \cref{recentwork}) in a universal way as follows.  

\begin{cor} \label{cor:dirichletcomp}
Suppose that $\langle a,a',b+r \rangle = R$.  Let $s,s',t\in R$ be such that  
\begin{equation} \label{eqn:sstsii}
s a+s' a'+t (b+r)=1.
\end{equation}

Then $f_1=e_1\otimes e'_1$ and $f_2=s e_1\otimes e'_2+s' e_2\otimes e'_1 + t e_2\otimes e'_2$ form an $R$-basis of $I\otimes_SI'$, and $\alpha+a(rs-ct), f_1, f_2$ is a good frame, yielding \begin{equation}
Q'' \simeq [aa', b-2a(ct-rs), cs'+(s b+ct)(c't+rs')+as(cs'+c's)].
\end{equation}
\end{cor}

\begin{proof}
Using \eqref{eqn:sstsii}, we compute the relations
\begin{align*}
(c't+rs')f_1+af_2 &=e_1 \otimes e'_2\\
(ct-rs)f_1+a'f_2&=e_2\otimes e'_1\\
-(cs'+c's)f_1+(b+r)f_2 &=e_2\otimes e'_2,
\end{align*}
then $f_1, f_2$ is an $R$-basis. Also, the relations \begin{align*}
    \alpha f_1&=(b-a(ct-rs))f_1-aa'f_2\\
    \alpha f_2&= (cs'+(sb+ct)(c't+rs')+as(cs'+c's))f_1+a(ct-rs)f_2
\end{align*} imply that $1,\alpha+a(rs-ct)$ and  $f_1, f_2$  is a good frame, which yields $Q''$.
\end{proof}

\begin{thm} \label{thm:composition}
    The set of isomorphism classes of pseudoregular $\scrO_Y$-modules carries a functorial, commutative partial monoid structure 
    $$[\scrI]\ast  [\scrI']=[\scrI\otimes_{\scrO_Y}\scrI']$$ defined whenever $\scrI$ and $\scrI'$ are coprimitive. 
\end{thm}

\begin{proof}
Well-definedness on isomorphism classes and functoriality are immediate. 
Closure under the tensor product follows from \Cref{tensor product is pseudoregular}. 
The identity element is $[\scrO_Y]$: indeed, $\scrO_Y\otimes_{\scrO_Y} \scrI\simeq \scrI$.
Commutativity, where defined, follows from $\scrI\otimes_{\scrO_Y} \scrI'\to \scrI'\otimes_{\scrO_Y} \scrI$.  Similarly, associativity holds whenever both sides are defined:  
$(\scrI\otimes \scrI')\otimes_{\scrO_Y} \scrI''\to I\otimes_{\scrO_Y}(I'\otimes_{\scrO_Y} \scrI'')$ is an $\scrO_Y$-module isomorphism, so we have $$([\scrI]\ast  [\scrI'])\ast [\scrI'']=[\scrI]\ast ([\scrI']\ast  [\scrI''])$$
as desired.
\end{proof}

\begin{definition}
We say two $\scrO_Y$-oriented binary quadratic $\scrO_X$-modules $Q,Q'$ are \defi{coprimitive} if the associated $\scrO_Y$-modules $\scrI=\Clf^1(Q)$ and $\scrI'=\Clf^1(Q')$ are coprimitive.  
\end{definition}

We unravel this locally. Let $S$ be a free quadratic $R$-algebra. Let $Q=[a,b, c]$ and $Q'=[a',b',c']$ be free, $S$-oriented, binary quadratic $R$-modules. The two $S$-orientations $\Clf^0(Q)\xrightarrow{\sim}S$ and $\Clf^0(Q')\xrightarrow{\sim}S$  determine an element $r \in R$ such that $b'=b+2r$ and $a'c'=ac+br+r^2$ as in \eqref{eqn:blerg}, and coprimitive is precisely the condition $\langle a,b,c,a',b',c',r\rangle = R$.  

\begin{proof}[Proof of \Cref{thm: composition law}]
Via the Clifford and norm functors, which induce an equivalence \Cref{Main Thm GSO} and therefore a bijection on isomorphism classes of objects, the partial composition law of \Cref{thm:composition} defined on isomorphism classes of coprimitive pseudoregular $\scrO_Y$-modules induces one on isomorphism classes of coprimitive $\scrO_Y$-oriented binary quadratic modules.
\end{proof}

\begin{remark}
Interestingly, our universal composition law works over a nontrivial open cover!  In particular, we have not shown that over a general commutative ring $R$, when $Q$ and $Q'$ are free binary quadratic modules over $R$, the composition is again free over $R$.  We wonder if in general there may indeed be a $K$-theoretic obstruction.
\end{remark}

\section{Rigidifications} \label{sec:rigid}

In this section, we pursue rigidifications (and their oriented versions), which then relate to quadratic modules up to isometry instead of similarity.

\subsection{Rigidifications} \label{sec:rigiddet}

The orientation defined in \cref{sec:orient} is quite different from that given by Wood \cite{Wood}, O'Dorney \cite{O'Dorney}, and Dallaporta \cite{Dallaporta}, as we now explain and put to a different purpose.  

We begin with the following motivating statement.  We recall (\cref{sec:framed-atlases}) that taking the value bundle gives a morphism $\catBQMod \to \catPic$.  We also recall the functor $\scrN \colon \catBPSReg \to \catPic$ \eqref{eqn:Nipic}.  

\begin{prop} \label{Main Thm sim-inpaper}
The equivalence $\catBQMod \xrightarrow{\sim} \catBPSReg$ in \eqref{catbqmodbpsreg}
factors over the stack $\catQuad \times \catPic$.  
\end{prop}

\begin{proof}
Given $Q \colon \scrM \to \scrL$ we get $\scrL$.  Clifford gives $\scrI=\scrM$ and $\scrO_Y=\Clf^0(Q)$; by the canonical isomorphism in \Cref{binary even cliff mod O_X},
\begin{equation} 
\scrN(\scrI) = \tbigwedge^2 \scrI \otimes (\scrO_Y/\scrO_X)^\vee \cong \tbigwedge^2 \scrM \otimes (\tbigwedge^2 \scrM \otimes \scrL^\vee)^\vee \cong \scrL
\end{equation}
so indeed this factors up to canonical isomorphism.  

Similarly, given $\scrI$ we get $\scrN(\scrI)$.  Applying the norm gives the quadratic module $N_\scrI \colon \scrI \to \scrN(\scrI)$ with value bundle exactly $\scrN(\scrI)$.  
\end{proof}

We may now look at $2$-fibers.  Let $\scrO_Y$ be a quadratic $\scrO_X$-algebra, let $\scrI$ be a rank-balanced $\scrO_Y$-module, and let $\scrL$ be an invertible $\scrO_X$-module.

\begin{definition}
An $\scrL$-\defi{rigidification} of $\scrI$ is an $\scrO_X$-module isomorphism $r\colon \scrN(\scrI)\to \scrL$. 
\end{definition}


Given an $\scrO_Y$-module isomorphism $\varphi\colon \scrI\to\scrI'$ where $\scrI, \scrI'$ are equipped with $\scrL$-rigidifications $r$ and $r'$ respectively, we say that $\varphi$ is \defi{rigidified} if the induced isomorphism $(\wedge^2\varphi)\otimes \id\colon \scrN(\scrI)\to \scrN(\scrI')$ commutes with the rigidifications, i.e., the diagram 
\begin{equation}
\begin{tikzcd}[column sep=large]
	\scrN(\scrI) & \scrN(\scrI') \\
	\scrL & \scrL
	\arrow["{(\wedge^2\varphi)\otimes \id}", from=1-1, to=1-2]
	\arrow["\wr", "r"', from=1-1, to=2-1]
	\arrow["r'", "\wr"', from=1-2, to=2-2]
    \arrow[equal, from=2-1, to=2-2]
\end{tikzcd}
\end{equation}
commutes. 

\begin{remark} \label{rmk:orientugh}
An $\scrL$-rigidification of $\scrI^\vee$ agrees with the $\scrL$-type orientation defined by Wood \cite{Wood}*{section 5}; over a Dedekind domain (where we consider a fractional ideal $\mathfrak{a}$ instead of $\scrL$), this also agrees with the type $\mathfrak{a}$ orientation defined by O'Dorney \cite{O'Dorney}*{section 2.2}. Dallaporta \cite{Dallaporta}*{section 3.1} defines an orientation of $\scrO_Y$ to be an $\scrO_X$-module isomorphism $\scrO_Y/\scrO_X\xrightarrow{\sim} \scrL^\vee$ which recovers the orientation in \cite{Wood}*{Theorem 5.2} by setting $\scrL=\scrO^\vee_X$. The choice of dual is to deal with discriminant and parities.  
\end{remark}

\begin{lem}\label{lem: rigidification is torsor}
    When nonempty, the set of $\scrL$-rigidifications of $\scrI$ is a torsor under $\Gm(X) \colonequals \scrO_X(X)^\times$.
\end{lem}

\begin{proof}
    The action of $u\in \Gm(X)$ on an $\scrL$-rigidification $r\colon \scrN(\scrI)\to \scrL$ is induced by the left multiplication $m_u$ by $u$ on $\scrL$, i.e., $u\cdot r=m_u\circ r$. Then $m_u\circ r=r$ forces $u=1$, and thus this implies freeness. Given another $\scrL$-rigidification $r'\colon \scrN(\scrI)\to \scrL$, let $v=r'\circ r^{-1} \in \Aut_{\scrO_X}(\scrL)\simeq \Gm(X)$ (because $\scrL$ is a line bundle). Then $v=m_t$ for a unique $t\in \Gm(X)$ and $r'=
    m_t\circ r$. This implies transitivity. 
\end{proof}



Given an $\scrL$-rigidification $r\colon \scrN(\scrI)\to \scrL$, we define the \defi{modified norm map} $N_{\scrI, \scrL}$ to be the composition $N_{\scrI, \scrL}=r\circ N_\scrI$.

\begin{proof}[Proof of \Cref{Main Thm SO}]
Combine \Cref{Main Thm sim-inpaper} with the factorization over $\catQuad$; then take $2$-fibers over $\scrO_Y$ and $\scrL$, applying \Cref{lem:fibres-of-equivalence}. 
\end{proof}

\begin{example}\label{exm:narrowclass}
If $Y=\Spec S$ over $X=\Spec \Z$ where $S$ is a real quadratic ring with discriminant $d$, then \Cref{Main Thm SO} recovers the classical statement of Gauss composition saying that primitive binary quadratic forms of discriminant $d$ up to $\SL_2(\Z)$-equivalence  are in bijection with the narrow class group of $S$, with the rigidification corresponding to a choice of generator of $S/\Z$.
\end{example}

\subsection{Comparison} \label{sec:sorryhereitis}

We conclude with some comparison with other definitions, recovering and refining results in other papers.  First, we define a
$\scrP$-\defi{determinant structure} on a locally free
$\scrO_X$-module $\scrE$ of finite rank $m$ to be an
$\scrO_X$-module isomorphism
\[
\omega \colon \tbigwedge^m \scrE \xrightarrow{\sim} \scrP.
\]
An $\scrL$-rigidification of $\scrI$ is equivalent to an
$\scrL\otimes(\scrO_Y/\scrO_X)$-determinant structure on $\scrI$.
Given an $\scrO_X$-module isomorphism
$\varphi \colon \scrE \to \scrE'$ where $\scrE,\scrE'$ have
$\scrP$-determinant structures $\omega,\omega'$, we define
$\det \varphi \in \scrO_X(X)^\times$ as usual.  

Returning to \Cref{rmk:orientugh}, one way to recover $\Pic Y$ is
by working with determinant structures on the quadratic algebra, as
follows.

Define the \defi{twisted value bundle} of
$Q\colon\scrM\to\scrL$ by
\[
\operatorname{tvb}(Q)
\colonequals
\scrL\otimes(\tbigwedge^2\scrM)^\vee.
\]
There is a canonical isomorphism $\operatorname{tvb}(Q) \simeq (\Clf^0(Q)/\scrO_X)^\vee$, so the functor $\operatorname{tvb}\colon\catQMod_2\to\catPic$
factors through $\Clf^0$.  

Let $\scrL_0$ be an invertible
$\scrO_X$-module.  An object in the $2$-fiber of
$\operatorname{tvb}$ over $\scrL_0$ is a quadratic module
$Q\colon\scrM\to\scrL$ together with an isomorphism
$\operatorname{tvb}(Q)\simeq\scrL_0$, or equivalently, after
transporting the value bundle, a quadratic map
\[
Q\colon\scrM\longrightarrow
\tbigwedge^2\scrM\otimes\scrL_0
\]
called an \defi{$\scrL_0$-twisted binary quadratic module}.
In the dual convention of linear binary quadratic forms, this is the
notion appearing in Wood \cite{Wood}*{Theorem 5.2}, and it agrees
with the definitions of Dallaporta
\cite{Dallaporta}*{Definition 3.8} and Mondal--Venkata Balaji
\cite{Mondal}*{Definition 5.14}.

A morphism between $\scrL_0$-twisted binary quadratic modules is an
$\scrO_X$-module isomorphism
$\varphi\colon\scrM\to\scrM'$ satisfying
\[
Q'(\varphi(x))
=
((\det \phi) \otimes 1)Q(x).
\]
Thus its similarity factor is prescribed by the determinant.

Let $\mathsf{OrAlg}_{2,\scrL_0}$ be the $2$-fiber over $\scrL_0$ of
the morphism $\catAlg_2 \to \catPic$ by $\scrO_Y \mapsto (\scrO_Y/\scrO_X)^\vee
\cong (\tbigwedge^2\scrO_Y)^\vee$.  An object is therefore a quadratic algebra $\scrO_Y$ equipped with an $\scrL_0^\vee$-determinant structure.  This is Dallaporta's notion of an \emph{$\scrL_0$-orientation}
\cite{Dallaporta}*{Definition 3.2}, also used by
Mondal--Venkata Balaji \cite{Mondal}*{Definition 5.12}; when
$\scrL_0=\scrO_X$, it is the orientation occurring in Wood
\cite{Wood}*{Theorem 5.2}.

Taking $2$-fibers over $\scrL_0$ in the Clifford--norm equivalence
gives an equivalence, now over $\mathsf{OrAlg}_{2,\scrL_0}$, between
$\scrL_0$-twisted binary quadratic modules under twisted
isomorphisms and pseudoregular modules over quadratic algebras
equipped with an $\scrL_0^\vee$-determinant structure, under
structure-preserving semilinear isomorphisms.  After restricting to
primitive objects and passing to isomorphism classes, this recovers
the correspondences of Wood \cite{Wood}*{Theorem 5.2}, Dallaporta
\cite{Dallaporta}*{Theorem 3.12}, and Mondal--Venkata Balaji
\cite{Mondal}*{Theorem 5.16}.

But we can go further by taking the $2$-fiber!  Let $\scrO_Y$ be a quadratic algebra with an $\scrL_0$-determinant structure, for example $\scrL_0=(\scrO_Y/\scrO_X)^\vee$ with the identity.  Taking the further $2$-fiber, on
the quadratic side, this consists of $\scrL_0$-twisted quadratic modules
equipped with a compatible $\scrO_Y$-orientation; on the module side, it is simply the category of pseudoregular $\scrO_Y$-modules under $\scrO_Y$-module isomorphisms.  Passing to isomorphism classes of primitive objects therefore gives $\Pic Y$.  Compare Dallaporta \cite{Dallaporta}*{Proposition 3.6 and
Theorem 3.21} and Mondal--Venkata Balaji
\cite{Mondal}*{Proposition 5.13 and Theorem 5.17}.  In those
class-level formulations, the hypothesis that $2$ be a
nonzerodivisor is used to show that an $\scrL_0$-oriented quadratic
algebra has no nontrivial automorphisms.  In the genuine $2$-fiber,
a chosen isomorphism of the even Clifford algebra with $\scrO_Y$ is
part of the data, so no such hypothesis is required.

\subsection{Rigidified Picard groups}\label{sec: rigidified Picard group}
In this section, we restrict similitude factors after fixing the value bundles.  

We first extend our generality a bit.  Let $H\leq \scrO_X(X)^\times$ be a subgroup.  For an $\scrO_Y$-module isomorphism $\varphi \colon \scrI \to \scrI'$ with rigidifications $r,r'$ as in \eqref{eqn:nlcis}, we say that $\varphi$ is \defi{$H$-rigidified} if the diagram
\begin{equation} \label{eqn:nlcis}
\begin{tikzcd}[column sep=large]
	\scrN(\scrI) & \scrN(\scrI') \\
	\scrL & \scrL
	\arrow["{(\wedge^2\varphi)\otimes \id}", from=1-1, to=1-2]
	\arrow["\wr", "r"', from=1-1, to=2-1]
	\arrow["r'", "\wr"', from=1-2, to=2-2]
    \arrow["u", from=2-1, to=2-2]
\end{tikzcd}
\end{equation}
commutes with $u \in H$.  So our previous notion of rigidification is the case $H=\{1\}$.  

Recall the group homomorphism $\scrN \colon \Pic Y \to \Pic X$ (\Cref{lem:toomanywedges}).
Let $\Pic_{\scrL} Y$ be the fiber of $\scrN\colon \Pic{Y}\to \Pic{X}$ over $[\scrL]\in \Pic{X}$.  Then $\Pic_{\scrL} Y$ is naturally torsor under the group $\Pic_{\scrO_X}Y=\ker{\scrN}$. 

Let $\RigPic_{\scrL,H} Y$ be the set of $\scrL$-rigidifed invertible $\scrO_Y$-modules under $H$-rigidified $\scrO_Y$-module isomorphism.

\begin{prop}\label{prop: fiber of Rigpic}
The forgetful map $$\pi\colon \RigPic_{\scrL, H} Y\to \Pic_\scrL Y$$ is a surjective map of pointed sets, and each fiber is a torsor for $$\scrO_X(X)^\times/(H\cdot\Nm \scrO_Y(Y)^\times).$$

When $\scrL=\scrO_X$, there is an exact sequence of groups
\begin{equation}
1 \to \scrO_X(X)^\times/(H\cdot\Nm \scrO_Y(Y)^\times) \to \RigPic_{\scrL, H} Y \to \Pic_{\scrL} Y \to 1.
\end{equation} 
\end{prop}

\begin{proof}
The forgetful map is indeed surjective; the fiber over $[\scrI]$ consists of $H$-rigidified isomorphism classes of $\scrL$-rigidifications on $[\scrI]$. Let $r, r'\colon \scrN(\scrI)\xrightarrow{\sim}\scrL$ be $\scrL$-rigidifications of $\scrI$. By \Cref{lem: rigidification is torsor}, there is a unique $u\in \scrO_X(X)^\times$ such that $r'=m_u\circ r$. Invertibility of $\scrI$ implies $\Aut_{\scrO_Y}(\scrI)^\times\simeq \scrO_Y(Y)^\times$, with each $\gamma\in \scrO_Y(Y)^\times$ acting as multiplication. The induced map $\wedge^2m_\gamma$ on $\wedge^2_{\scrO_X}\scrI$ is multiplication by $\Nm(\gamma)\in\scrO_X(X)^\times$. Thus, $m_\gamma$ is $H$-rigidified if and only if $u\Nm(\gamma)=r'\circ (\wedge^2m_\gamma\otimes\id)\circ r^{-1}\in H$. This implies that $(\scrI, r)$ and $(\scrI, r')$ are in the same $H$-rigidified $\scrO_Y$-module isomorphism class if and only if $u\in H\cdot\Nm\scrO_Y(Y)^\times$. It follows that the fiber $\pi^{-1}([\scrI])$ is in bijection with $\scrO_X(X)^\times/(H\cdot\Nm \scrO_Y(Y)^\times)$, which acts on the fiber by $[u]\cdot[(\scrI, r)]=[(\scrI, m_u\circ r)]$, so the fiber is a torsor for this group.

When $\scrL=\scrO_X$, the tensor product over $\scrO_Y$ endows $\RigPic_{\scrO_X, H}Y$ with a group structure: given $[(\scrI, r)]$ and $[(\scrI', r')]$, we define
$$[(\scrI, r)]\cdot [(\scrI', r')]=[(\scrI\otimes_{\scrO_Y}\scrI', r\cdot r')]$$ where $r\cdot r'$ is the composite
\begin{equation}
    \scrN(\scrI\otimes_{\scrO_Y}\scrI')\xrightarrow{\sim}\scrN(\scrI)\otimes_{\scrO_X}\scrN(\scrI')\xrightarrow{r\otimes r'}\scrO_X\otimes_{\scrO_X}\scrO_X\xrightarrow{\sim}\scrO_X
\end{equation}
via the canonical isomorphism \eqref{eqn:secondstatni}. The identity is $[(\scrO_Y, r_0)]$ where $r_0\colon \scrN(\scrO_Y)\to \scrO_X$ is the canonical isomorphism in \Cref{exm:ournormisanorm}. It is straightforward to check that this is well-defined on the $H$-rigidified $\scrO_Y$-module isomorphism classes and $\pi$ is a group homomorphism. 

It remains to identify the kernel of $\pi$. An element $[(\scrI, r)]$ lies in the kernel if and only if $[\scrI]=[\scrO_Y]$. Choose an $\scrO_Y$-module isomorphism $\varphi\colon \scrI\to \scrO_Y$, and this induces an isomorphism $\scrN(\varphi)\colon \scrN(\scrI)\xrightarrow{\sim}\scrN(\scrO_Y)$. Then the rigidification $r$ can be written as $r=m_{u_r}\circ r_0\circ\scrN(\varphi)$ for a unique $u_r\in\scrO_X(X)^\times$. Another $\varphi'\colon\scrI\xrightarrow{\sim}\scrO_Y$ differs by $\varphi'=m_\gamma\circ \varphi$ for some $\gamma\in\scrO_Y(Y)^\times$, and $r=m_{u'r}\circ r_0\circ \scrN(\varphi')$ for another $u'_r\in\scrO_X(X)^\times$. This gives $$r_0\circ\scrN(\varphi')=r_0\circ m_{\Nm(\gamma)}\circ\scrN(\varphi)=m_{\Nm(\gamma)}\circ r_0\circ \scrN(\varphi).$$ Thus, $u'_r=u_r\Nm(\gamma)^{-1}$, so $u_r$ is well-defined modulo $\Nm\scrO_Y(Y)^\times$.

We claim that the map
\begin{align*}
\rho\colon\ker{\pi}&\to\scrO_X(X)^\times/(H\cdot\Nm\scrO_Y(Y)^\times)\\
(\scrI,r)&\mapsto u_r 
\end{align*}
is well-defined. Consider another $\scrO_X$-rigidification $r'$ of $\scrI$ and an $H$-rigidified $\scrO_Y$-module isomorphism $\psi\colon \scrI\to \scrI$ with respect to $r$ and $r'$. Then we have $\wedge^2\psi=m_{\Nm(\gamma)}$ for some $\gamma\in\scrO_Y(Y)^\times$, and thus $$r'\circ(\wedge^2\psi)\circ r^{-1}=(m_{u_{r'}}\circ r_0\circ\scrN(\varphi))\circ m_{\Nm(\gamma)}\circ(m_{u_r}\circ r_0\circ \scrN(\varphi))^{-1}=u_{r'}\Nm(\gamma)u^{-1}_r\in H.$$ This implies that $\rho$ is well-defined. The group structure of $\RigPic_{\scrO_X, H}Y$ guarantees that $\rho$ is a group homomorphism.

For surjectivity, given $u\in\scrO_X(X)^\times$, we have $\rho([(\scrO_Y, m_u\circ r_0)])=[u]$. For injectivity, suppose that $\rho([\scrI, r])=0$.  This implies $u_r=h\Nm(\gamma)$ for some $h\in H$ and $\gamma\in \scrO_Y(Y)^\times$. Taking $\psi=m_\gamma\circ\varphi$ for an $\scrO_Y$-module isomorphism $\varphi\colon \scrI\simeq\scrO_Y$ gives $$r_0\circ\scrN(\psi)\circ r^{-1}=m_{\Nm(\gamma)}\circ m^{-1}_{u_r}=h^{-1}\in H.$$ Hence, $\psi$ is an $H$-rigidified isomorphism from $(\scrI, r)$ to $(\scrO_Y, r_0)$, so $[(\scrI, r)]=[(\scrO_Y, r_0)]$. 
\end{proof}

\begin{prop}\label{prop: H-similitude}
The Clifford and modified norm functors define a discriminant-preserving equivalence of categories fibered over $\catQuad \times_{\catSch} \catPic$ between 
\begin{center}
$\scrO_Y$-oriented binary quadratic $\scrO_X$-modules under oriented $H$-similitudes
\end{center}
and
\begin{center}
rigidified pseudoregular $\scrO_Y$-modules under $H$-rigidified $\scrO_Y$-module isomorphisms.
\end{center}
\end{prop}

\begin{proof}
The group $H$ acts naturally on both sides of \Cref{Main Thm SO}; the bijection comes from taking the $2$-fiber and $H$ acts on $\scrL$, so the bijection is $H$-equivariant.  Taking the quotient gives the result.  
\end{proof}

\begin{cor}\label{cor: H-sim classe =RigPic}
The Clifford and modified norm functors define mutually inverse bijections between the set of $H$-similitude classes of primitive $\calO_Y$-oriented binary quadratic $\scrO_X$-modules with value module $\scrL$ and the set $\RigPic_{\scrL, H} Y$.  
\end{cor}

\begin{proof}
Immediate from taking isomorphism classes on both sides of \Cref{prop: H-similitude}.
\end{proof}

We now assemble the family of torsors $\RigPic_{\scrL,H} Y \to \Pic_{\scrL} Y$.
Concretely, choose a representative $\scrL$ for each isomorphism class $[\scrL]\in\Pic X$. We define 
\begin{equation}
    \RigPic_H{Y}\colonequals\bigsqcup_{[\scrL]\in\Pic{X}}\RigPic_{\scrL, H}Y.
\end{equation}

\begin{lem}
The set $\RigPic_H Y$ is well-defined up to bijection over $\Pic Y$.  
\end{lem}

\begin{proof}
Another choice $\scrL' \xrightarrow{\sim} \scrL$ induces a bijection on $\RigPic_{\scrL, H} Y$, and this obviously commutes with the map to $\Pic Y$. 
\end{proof}

Forgetting the rigidification gives an exact sequence of pointed sets
\begin{equation} 
1 \to \scrO_X(X)^\times/(H\cdot\Nm \scrO_Y(Y)^\times) \to \RigPic_H Y \to \Pic Y \to 1.
\end{equation}
Indeed, the last map is surjective (every $\scrI$ has some rigidification); the elements of $\RigPic Y$ that map to the identity are exactly the isomorphism classes of rigidifications of $\scrO_Y$, which as in \Cref{prop: fiber of Rigpic} are in the image of $\scrO_X(X)^\times/\Nm \scrO_Y(Y)^\times$.  

Despite this nice structure, it is a bit of a dicey proposition to give the middle pointed set the structure of a group.  To do so, we define a \defi{normalized system (of representatives)} for $\Pic{X}$ to be the data of 
\begin{itemize}
    \item a representative $\scrL_s$ for each $s\in\Pic{X}$, with $\scrL_0=\scrO_X$ and
    \item transport isomorphisms $l_{s,s'}\colon \scrL_s\otimes \scrL_{s'}\xrightarrow{\sim}\scrL_{s+s'}$;
\end{itemize}
satisfying the two conditions:
    \begin{enumroman}
        \item $l_{0,s}$ and $l_{s,0}$ are the canonical identifications $\scrO_X\otimes_{\scrO_X}\scrL_s\xrightarrow{\sim}\scrL_s$ and $\scrL_s\otimes_{\scrO_X}\scrO_X\xrightarrow{\sim}\scrL_s$ respectively; and 
        \item $l_{s+s', s''}\circ(l_{s,s'}\otimes\id_{\scrL_{s''}})=l_{s, s'+s''}\circ(\id_{\scrL_s}\otimes l_{s',s''})$.
    \end{enumroman}

\begin{lem}
    A normalized system 
    for $\Pic X$ always exists. 
\end{lem}
\begin{proof}
This follows from a stronger result by Johnson--Osorno \cite{JO}*{Theorem 2.2} by choosing a skeleton for $\Pic X$. We provide a sketch of an explicit construction of a normalized system.

Choose invertible $\scrO_X$-modules $\scrL_i$ whose classes generate $P \colonequals \Pic{X}$, let $A \colonequals \bigoplus_s \Z e_s$, with the natural projection map $\pi\colon A \to P$ sending $e_s \to \scrL_s$.  We construct a normalized system over the free group $A$ by taking for $a=\sum_s a_s e_s$ the module $\scrF_a \colonequals \bigotimes_s \scrL_s^{a_s}$.  Let $N \colonequals \ker \pi$, choose a basis $\{n_j\}_j$ for $N$; for every $n_j$ in the basis choose a trivialization $t_{n_j} \colon \scrF_{n_j} \xrightarrow{\sim} \scrO_X$ and extend these multiplicatively to give trivializations $t_n \colon \scrF_n \to \scrO_X$ for each $n \in N$.  Choose a set-theoretic splitting $\sigma \colon P \to A$.  We get transport isomorphisms from the composite
\[ \scrL_s \otimes \scrL_{s'} = \scrF_{\sigma(s)} \otimes \scrF_{\sigma(s')} \xrightarrow{\sim} \scrF_{\sigma(s)+\sigma(s')} \xrightarrow{t_{\sigma(s+s')-\sigma(s)-\sigma(s')}} \scrF_{\sigma(s+s')} = \scrL_{s+s'}. \]
It is straightforward but elaborate to check compatibility, using the cocycle identity coming from the equality $\sigma(s)+\sigma(s')+\sigma(s'') - \sigma(s+s'+s'')$ written in two ways.  
\end{proof}

\begin{cor} \label{cor: rigpic is group}
Any choice of normalized system endows the set $\RigPic_H{Y}$ with a group structure such that it fits into an exact sequence 
    \begin{equation}\label{eq: Rigpic H exact seq}
1 \to \scrO_X(X)^\times/(H\cdot\Nm \scrO_Y(Y)^\times) \to \RigPic_{H} Y \to \Pic Y \to 1.
\end{equation} 
\end{cor}
\begin{proof}
Let $\{\scrL_s, l_{s,s'}\}_{s, s'}$ be a normalized system. Given $[(\scrI, r)]\in\RigPic_{\scrL_s, H}{Y}$ and $[(\scrI', r)]\in\RigPic_{\scrL_{s'}, H}{Y}$, the canonical isomorphism  \eqref{eqn:secondstatni} gives a rigidification $r\cdot r'$ with value bundle $\scrL_s\otimes_{\scrO_X}\scrL_{s'}$. Transporting along the isomorphism $l_{s,s'}$ to the representative $\scrL_{s''}$ of $[\scrL_s\otimes_{\scrO_X}\scrL_{s'}]$, we obtain $[\scrI\otimes_{\scrO_Y}\scrI', r\cdot r']\in \RigPic_{\scrL_{s''}, H}Y$. The normalization conditions $\scrL_0=\scrO_X$, together with the fact that $l_{0,s}$ and $l_{s,0}$ are the canonical identifications, imply that the unit is the class of $(\scrO_Y, r_0)$ in the fiber over $0$. The associativity condition in the definition of normalized system implies that this multiplication is associative. Therefore, $\RigPic_H Y$ is a group. 

The forgetful map $\RigPic_H Y\to \Pic Y$ is a surjective group homomorphism since this holds fiberwise by \Cref{prop: fiber of Rigpic}. Its kernel is exactly the fiber over the trivial class, which is identified again with $\scrO_X(X)^\times/(H\cdot \Nm{\scrO_Y}(Y)^\times)$.
\end{proof}

\begin{remark}
The group structure on $\RigPic_H Y$ constructed in \Cref{cor: rigpic is group} depends on the choice of normalized system, equivalently the group extension.  As usual, this is governed by the cohomology group $H^2(\Pic X, \scrO_X(X)^\times/(H\cdot \Nm\scrO_Y(Y)^\times))$.  
\end{remark}

\begin{exm}
One way to resolve the ambiguity in the definition of $\Pic Y$ is to take $H=\scrO_X(X)^\times$. Then we have simply $\RigPic Y \xrightarrow{\sim} \Pic Y$ canonically.  Over a Dedekind domain, O'Dorney \cite{O'Dorney} implicitly performs the same assembly as \Cref{cor: rigpic is group} used to derive the group law on classes of binary quadratic forms in Bhargava's reinterpretation of Gauss composition (see \cite{O'Dorney}*{Theorem 5.4 and Corollary 5.6}), with an equivalent notion of orientation (cf.\ \Cref{rmk:orientugh}).  See also Wood \cite{Wood}*{section 5} who similarly proposes assembling classes over varying value modules in this manner.  
\end{exm}






\section{Applications}\label{sec: app}
We record two important consequences of the preceding equivalences. 

\subsection{Over Dedekind domains}

Let $X=\Spec R$ with $R$ a Dedekind domain.  Let $F \colonequals \Frac R$ be the field of fractions of $R$.  Let $Q \colon V \to F$ be a nondegenerate binary quadratic form and let $K \colonequals \Clf^0(Q)$ be its even Clifford algebra, a quadratic $F$-algebra.  Then $K$ is a field or $K \simeq F \times F$.  Let $S\subset K$ be an $R$-order.

\begin{cor}
The Clifford and norm functors define mutually inverse discriminant-preserving bijections as follows:
\begin{enumalph}
\item between the set of similarity classes of $R$-lattices $M \subset V$ with multiplicator ring $S$ and the group $\Pic S$; and
\item between the set of isometry classes of integral $R$-lattices $M \subseteq V$ with multiplicator ring $S$ and $Q(M)R=R$ and the group $\RigPic_R S$. 
\end{enumalph}
\end{cor}

By functoriality, these bijections commute with base change, including localization and completion.  

\begin{proof}
An $R$-lattice $M\subset  V$ with multiplicator ring $S$ determines a binary quadratic $R$-module $Q_M\colon M\to L_M$, where $L_M\coloneq Q(M)R$ is the fractional ideal generated by the values of $Q$ on $M$. By construction, $Q_M$ is primitive. We claim that $\Clf^0(Q_M)=S$. Indeed, it is immediate that $\Clf^0(Q_M)\subseteq S$ since $M$ is a $\Clf^0(Q_M)$-module. In the other direction, we work locally with $M=Re_1\oplus Re_2$ and $Q_M=[a,b,c]$, so $\Clf^0(Q_M)=R[\alpha]$ where $\alpha=e_1e_2$ and $\alpha^2-b\alpha+ac=0$. Taking $k=r+s\alpha\in S$ with $r,s\in F$, we have $ke_1=re_1+s(b-e_2e_1)e_1=(r+sb)e_1-sae_2\in M$ and $ke_2=sce_1+re_2\in M$. This implies that $r, sa, sb, sc\in R$. Since $Q_M$ is primitive, there exist $u,v,w\in R$ such that $au+bv+cw=1$, so $s=(sa)u+(sb)v+(sc)w\in R$. Thus, we have $S\subseteq \Clf^0(Q_M)$ and conclude that $S=\Clf^0(Q_M)$ globally.

Since every $R$-lattice has a canonical orientation obtained by restricting the canonical orientation on $V$, the oriented formalism of \Cref{GSO class =Pic Y} applies. Passing to similarity classes, we obtain the claimed bijection with $\Pic S$. Further restricting to the case of integral lattices with $L_M=R$ proves (b) by \Cref{Main Thm SO}.
\end{proof}  

In this setting, \eqref{eq: Rigpic H exact seq} becomes
\begin{equation}
1 \to R^\times/\Nm S^\times \to \RigPic S \to \Pic S \to 1.
\end{equation} We recognize $R^\times/\Nm S^\times \cong \widehat{H}^0(\Gal(K\,|\,F),S^\times)$ as Tate cohomology.  Concretely, $\RigPic_R S$ is the group of pairs $(I,a)$ with $I$ an invertible fractional $S$-ideal and $a \in F^\times$ satisfying $\Nm I = a R$ modulo the principal subgroup $\{(\alpha S,\Nm \alpha) : \alpha \in K^\times\}$.  

We recover Gauss composition as follows.  For $S=\Z_K$ the ring of integers of a real quadratic field over $R=\Z$,  we recover $\RigPic \Z_K \cong \Cl^+ \Z_K$ the narrow class group.  For $K$ an imaginary quadratic, $\RigPic \Z_K \cong \Cl \Z_K \times \{\pm 1\}$; on the level of binary quadratic forms, this corresponds to allowing forms that are either positive or negative definite.  In other words, considering isometries instead of similarities gives generalizations of narrow class groups.
\subsection{Genera and class sets}\label{genus}

In this section, we consider isomorphisms of objects locally at all points in $X$. 

For a quadratic $\scrO_X$-module $(\scrM, \scrL, Q)$ and a point $x \in X$, let $(\scrM_x, \scrL_x, Q_x)$ be the base extension of $Q$ with respect to $\scrO_{X,x}$ where $\scrM_x=\scrM \otimes_{\scrO_X} \scrO_{X,x}$ and $\scrL_x=\scrL \otimes_{\scrO_X} \scrO_{X,x}$. If $Q$ has an $\scrO_Y$-orientation, then the corresponding local orientation for $Q_x$ is $(\scrO_Y)_x$ for $x\in X$.

Let $Q$ be $\scrO_Y$-oriented and $H\leq \scrO_X(X)^\times$.
\begin{itemize}
\item We define the \defi{oriented similarity genus} of $Q$ to be the set $\Gen_{\GSO}(Q)$ of quadratic $\scrO_X$-modules $(\scrM', \scrL', Q')$ such that there is an $(\scrO_Y)_x$-oriented similarity between $Q_x$ and $Q'_x$ for all $x\in X$.  
\item We can similarly define the \defi{oriented $H$-similitude genus} $\Gen_{\GSO, H}(Q)$. 
\item Forgetting orientations, we can also define the \defi{similarity genus} $\Gen_{\GO}(Q)$ and the \defi{$H$-similitude genus} $\Gen_{\GO, H}(Q)$. 
\end{itemize}
We define the \defi{oriented similarity class set} $\Cls_{\GSO}(Q)$ and  \defi{similarity class set} $\Cls_{\GO}(Q)$, as well as the \defi{oriented $H$-similitude class set} $\Cls_{\GSO,H}(Q)$ and the \defi{$H$-similitude class set} $\Cls_{\GO,H}(Q)$, to be the sets of global classes in the corresponding genera respectively. 

\begin{thm}\label{GSO class=Cls}
 All primitive $\scrO_Y$-oriented binary quadratic $\scrO_X$-modules are in the same oriented similarity genus. 
\end{thm}

\begin{proof}
    It is enough to show that for any primitive $\scrO_Y$-oriented binary quadratic $\scrO_X$-module $(\scrM, \scrL, Q)$, there is an $(\scrO_Y)_x$-oriented similarity between $Q_x$ and the norm $\Nm_x\colon (\scrO_Y)_x\to \scrO_{X,x}$. Given $x\in X$, we have $(\scrO_Y)_x=\scrO_{X,x} \oplus \scrO_{X,x}\alpha$, $\scrM_x=\scrO_{X,x}e_1\oplus\scrO_{X,x}e_2$ for a good frame $\alpha, e_1, e_2$ of $(\scrO_Y)_x$ and $\scrM_x$. We also have $\scrL_x=\scrO_{X,x}l$ for some $l\in \scrL_x$. By \Cref{prop:canorient}, there is a unique $\scrO_{X,x}$-algebra isomorphism between $\Clf^0(Q_x)$ and $(\scrO_Y)_x$ given by the map $e_1e_2\mapsto \alpha$. Suppose that $Q_x(xe_1+ye_2)=(ax^2+bxy+cy^2)l$ for $a,b,c\in\scrO_{X,x}$. Then we have $(e_1e_2)^2=be_1e_2-ac$, and thus $\Nm_x(x+y\alpha)=x^2+bxy+acy^2$. Since $Q_x$ is primitive and $\scrO_{X,x}$ is a local ring, there exists some $x_0, y_0$ such that $ax^2_0+bx_0y_0+cy^2_0\in \scrO^\times_{X,x}$, and we can replace $e_1$ by $x_0e_1+y_0e_2$ and extend it to a basis for $\scrM_x$. Thus, without loss of generality, we may assume that $a$ is a unit. Let $\varphi \colon \scrM_x\to (\scrO_Y)_x$ and $\psi\colon \scrL_x\to \scrO_{X,x}$ be the $\scrO_{X,x}$-module isomorphisms defined by $\varphi(xe_1+ye_2)=x+a^{-1}y\alpha$ and $\psi(rl)=a^{-1}r$. Then we have $\Nm_x(\varphi(xe_1+ye_2))=x^2+a^{-1}bxy+a^{-1}cy^2=\psi(Q_x(xe_1+ye_2))$, and thus $(\varphi, \psi)$ is a similarity between $Q_x$ and $\Nm_x$. The choice of good frame and \Cref{prop:canorient} ensure that the similarity is oriented. 
\end{proof}

\begin{cor}\label{GSO class set}
    The oriented similarity classes of primitive $\scrO_Y$-oriented binary quadratic $\scrO_X$-modules are exactly the elements of $\Cls_{\GSO}(Q)$ for any primitive $\scrO_Y$-oriented binary quadratic $\scrO_X$-module $Q$.
\end{cor}

\begin{proof}
    Taking global classes in \Cref{GSO class=Cls} yields the result. 
\end{proof}






\appendix

\section{Recent work}\label{recentwork}

In this appendix, we review recent work generalizing Gauss composition. The history of the composition of binary quadratic forms and the revisiting of Gauss composition in the 19th and early 20th centuries can be found in \cite{Dickson23} (see also \cite{Towber}*{Introduction}).

Gauss's original approach in \cite{Gauss} was complicated. Influenced by Legendre, Dirichlet and Dedekind took another route in \cite{Dirichlet}*{Supplement X} (see also \cite{Cox}*{Section 3A} or \cite{Dickson51}*{Section IX}). Two quadratic forms $ax^2+bxy+cy^2$ and $a'x^2+b'xy+c'y^2$ over $\Z$ are called \emph{united} if they have the same discriminant and $\gcd(a, a', \frac{b+b'}{2})=1$. Dirichlet and Dedekind defined a composition for united forms and realized a group structure for the set of classes of primitive positive definite forms of the same discriminant. In 1961, following Dirichlet and Dedekind, Lubelski \cite{Lubelski} generalized the composition over Euclidean rings. 

In 1968, Butts--Estes \cite{ButtsEstes} introduced $C$-domains $D$, which are integral domains with characteristic not 2 and certain structural conditions, and developed explicit composition laws for binary quadratic forms over $D$ via united forms. They showed that the set of classes of primitive binary quadratic forms over $D$ with discriminants in a certain set $\Delta$ is a commutative semigroup. If the discriminant is fixed, then this recovers the classical case. Moreover, they considered the quadratic algebra $R=D[\omega]$ and gave a norm condition for an $R$-ideal generated by $A$ and $b+\omega$ as a $D$-module to be invertible where $A$ is an invertible ideal of $D$ and $b\in D$. The direct compounds of binary quadratic forms correspond exactly to products of free rank-two $R$-modules, and whenever products of free $D$-modules in $R$ remain free, the classes of primitive forms with discriminant $d\in\Delta$ form an abelian group $G_d$ under composition. They also constructed a surjective homomorphism from $G_d$ to a subgroup of the ideal class group of $R$, identifying the kernel as those classes containing a form that represents a unit in $D$.

Also in 1968, Kaplansky \cite{Kaplansky} extended the theory of composition to a Bézout domain of characteristic not 2 via module multiplication. This recovers the same composition of united forms. 

In 1972, Butts--Dulin \cite{ButtsDulin} studied the connection between Gauss composition and composition of united forms. They considered the compound and the Gaussian compound of binary quadratic forms defined by Gauss. Then they gave necessary and sufficient conditions for existence in each case over a Bézout domain and for existence of composition of united forms in an integral domain of characteristic not 2 where two primitive binary quadratic forms of the same discriminant have a Gaussian compound. Furthermore, they showed that the composition of united forms holds over a Bézout domain.

In 1980, Towber  \cite{Towber} gave a generalization in which the domain of a quadratic form $Q\colon M\to R$ where $M$ is a locally free module of rank 2 over a commutative ring $R$ (with $1$) such that $\tbigwedge^2 M$ (equipped with an orientation given by a choice of generator) is a free $R$-module of rank 1. He also introduced a composition law on the set of equivalence classes under the action of $\SL_2(R)$ of primitive, oriented binary quadratic forms $ax^2+bxy+cy^2$ with a given discriminant and a given residue class of $b$ in $R/2R$. 

In 1982, Kneser \cite{Kneser} considered quadratic forms $Q \colon M \to R$ over a commutative ring $R$ where $\Clf^0(Q)$ is isomorphic to a given \emph{type} which is an $R$-algebra $C$. This $C$-module structure ensures the existence and the uniqueness of a composition map. The isomorphism classes of primitive binary quadratic forms of type $C$ with the composition law form an abelian group $G(C)$. Forgetting the quadratic form on $M$ yields a homomorphism $G(C)\to \Pic(C)$ which is in general neither injective nor surjective. To obtain the full Picard group, he then considered $M$ as invertible $C$-modules. \cite{Kneser}*{Proposition 2} gives quadratic $R$-modules (also called \emph{quadratic maps} by Kneser) $Q\colon M \to N$ where $N$ is an invertible $R$-module, $Q(cx)=\Nm_{C/R}(c)Q(x)$ for all $c\in C$ and $x\in M$, and $Q$ is primitive. The isomorphism classes (oriented similarity classes in our notion) of such quadratic modules are isomorphic to $\Pic(C)$. This is exactly recovered by \Cref{GSO class =Pic Y}.

Most of the works above only considered $\SL_2$-equivalence of binary quadratic forms. In 2000, Mastropietro \cite{Mastropietro} considered equivalences given by matrices of determinants which are totally positive units over real quadratic number fields of class number
one with totally complex quadratic extensions. He gave the construction of a correspondence between ideal class groups and such equivalence classes.

In 2004, Bhargava \cite{Bhargava1} considered the space of $2 \times 2 \times 2$ cubical integer matrices modulo the natural action of $\SL_2(\Z) \times \SL_2(\Z) \times \SL_2(\Z)$ and described six composition laws, including Gauss composition, derived from this approach. He then gave interpretations of the connection between orbits of the six spaces and ideal classes of quadratic orders.  (It would be interesting to pursue a sheafified version of this approach.) \cite{Bhargava2} and \cite{Bhargava3} in the same year and \cite{Bhargava4} in 2008 developed analogous laws of composition on forms of degree $k>2$ so that the resulting orbits parametrize orders in number fields of degree $k$.  

In 2011, Wood \cite{Wood} gave a generalization of Gauss composition over any base scheme. Her dual point of view considered linear binary quadratic forms over a scheme $X$ as global sections of $\Sym^2(\scrM)\otimes \scrL$ where $\scrM$ is a locally free $\scrO_X$-module of rank 2 and $\scrL$ is an invertible $\scrO_X$-module (see discussion in \cref{sec:wood}). The $\GL_2(\Z)\times \GL_1(\Z)$-equivalences are exactly our notion of similarities. On the other hand, she considers pairs $(\scrO_Y, \scrI)$ where $\scrO_Y$ is a quadratic $\scrO_X$-algebra and $\scrI$ is a traceable $\scrO_Y$-module. \Cref{prop:traceable} has shown that being traceable is equivalent to being pseudoregular. Then \Cref{Main Thm GO} recovers \cite{Wood}*{Theorem 1.4}, and thus restricting to primitive forms recovers \cite{Wood}*{Theorem 1.5} because primitive forms correspond to invertible modules by \Cref{primitive = invertible}. Wood \cite{Wood}*{section 5} further equips an $\scrL$-type orientation corresponding to $(\scrO_Y, \scrI)$, and \Cref{Main Thm SO} recovers \cite{Wood}*{Theorem 5.1}. Unfortunately, Wood did not give group laws or isomorphisms with Picard groups in our style; instead, her results only provided a set-theoretic bijection with a disjoint union of quotients of Picard groups in the primitive case.  

Wood mentioned that Lenstra suggested in a talk that Kneser's approach could yield a theorem in the style of \cite{Wood}*{Theorem 5.1} in the case when the forms are primitive and nondegenerate. He also suggested a Clifford algebra construction which could provide a theorem along the lines of \cite{Wood}*{Theorem 1.4}.

In 2016, O'Dorney \cite{O'Dorney} generalized Bhargava's theory of higher composition law over arbitrary Dedekind domains. The parametrization of quadratic, cubic, and quartic algebras as well as ideal classes in
quadratic algebras extended Bhargava's reinterpretation of Gauss composition. 

In 2021, Zemková \cite{Zemkova} generalized Mastropietro's work to the case where the base field is any number field of narrow class number one and gave an explicit correspondence between the ideal class group and the equivalence classes of binary quadratic forms. She also gave a short overview of previous works in the introduction section. 

In 2025, following Wood's approach with extra refinements, Dallaporta
\cite{Dallaporta} recovered from Wood’s bijection an explicit bijection between the
Picard group of a given quadratic algebra and a set of classes of primitive quadratic
forms over a scheme. \Cref{GSO class =Pic Y} recovers \cite{Dallaporta}*{Theorem 1.2} where the corresponding quadratic algebra is specified by the discriminant and parity.  In this work, the author often assumes that $2$ is a nonzerodivisor on the base.  

Also in 2025, Bitan \cite{Bitan} considered Gauss composition for binary quadratic forms over the ring of regular functions over
an affine curve over a finite field of odd characteristic. He used \'etale cohomology to describe the bijection between the equivalence classes of binary quadratic forms and the Picard group of the corresponding quadratic algebra. 

Finally and very recently, Mondal \cite{Mondal-thesis} and Mondal--Venkata Balaji \cite{Mondal} also adopted the point of view that Gauss composition can be seen through the lens of the Clifford functor.  They prove functorial bijections (but not an equivalence of categories) for similarities and show that for similitude groups one can recover the Picard group when $2$ is a nonzerodivisor.  To show that the Clifford map is bijective on classes, they also applied the method of universal norms due to Bichsel--Knus \cite{BK}.

\end{document}